\documentclass[reqno, 10pt]{amsart}
\usepackage[centertags]{amsmath} \usepackage{amsfonts}
\usepackage{amssymb} \usepackage{amsthm}
\usepackage{array, textcomp, color, colortbl, float, graphicx, mathrsfs, latexsym, arydshln, multirow, xcolor}
\usepackage[all]{xy}
\usepackage{newlfont}
\theoremstyle{definition}\newtheorem{theorem}{Theorem}[section]
\theoremstyle{definition}
\theoremstyle{definition}\newtheorem{proposition}{Proposition}[section]
\theoremstyle{definition}
\theoremstyle{definition}
\theoremstyle{definition}\newtheorem{illustration}{Illustration}[section]
\theoremstyle{definition}

\newcommand{\gammaij}[2]{\gamma_{#1, \: #2}}
\newcommand{\Fij}[2]{F_{#1, \: #2}}
\newcommand{\Oij}[2]{\mathcal{O}_{#1, \: #2}}
\newcommand{\im}[2]{\text{Im}(#1, \: #2)}

\definecolor{lightyellow}{rgb}{1.0, 1.0, 0.6}

\renewcommand{\footnote}{\endnote}
\newcommand{\ignore}[1]{}\makeglossary

\calclayout

\begin{document}

\title{Geometric realizations of regular abstract polyhedra with automorphism group $H_3$}
%\shorttitle{Geometric realizations of regular abstract polyhedra}
\author{Jonn Angel L. Aranas and Mark L. Loyola}
%\cauthor[]{Mark L.}{Loyola}{mloyola@ateneo.edu}{address if different from \aff}
%\aff[]{Department of Mathematics, Ateneo de Manila University, Katipunan Avenue, Loyola Heights, Quezon City 1108, \country{Philippines}}
\footnotetext[1]{This work was supported by the School of Science and Engineering Industry 4.0 Research Grant (Project Code SI4-017) of Ateneo de Manila University. J.A.L. Aranas is grateful to the DOST-SEI ASTHRDP-NSC for his graduate scholarship grant.} 
\date{25th of October, 2019}
\maketitle

\begin{abstract}
    A \textit{geometric realization} of an abstract polyhedron $\mathcal{P}$ is a mapping $\rho : \mathcal{P} \to \mathbb{E}^3$ that sends an $i$-face to an open set of dimension $i$. This work adapts a method based on Wythoff construction to generate a full rank realization of a regular abstract polyhedron from its automorphism group $\Gamma$. The method entails finding a real orthogonal representation of $\Gamma$ of degree 3 and applying its image to suitably chosen open sets in space. To demonstrate the use of the method, we apply it to the abstract polyhedra whose automorphism groups are isomorphic to the non-crystallographic Coxeter group $H_3$.
\end{abstract}

\section{Introduction}
\label{sec:intro}

    A \textit{geometric polyhedron} is typically described as a three-dimensional solid of finite volume bounded by flat regions called its \textit{facets}. Well-known examples of geometric polyhedra include the five Platonic solids, which have been studied since antiquity, and their various \textit{truncations} and \textit{stellations} (Coxeter, 1973). Because of their mathematical and aesthetic appeal, geometric polyhedra are widely used as models in various fields of science and the arts (Senechal, 2013). In the field of crystallography, they have been used in studying the symmetry and structural formation of crystalline materials (Schulte, 2014; Delgado-Friedrichs \& O'Keefe, 2017), nanotubes (Cox \& Hill, 2009, 2011), and even viruses (Salthouse \textit{et al.}, 2015).

    In classical geometry, a facet is a convex or a star polygon bounded by line segments called \textit{edges} and corner points called \textit{vertices}. The facets enclose an open region in space called the polyhedron's \textit{cell}. Modern treatments of geometric polyhedra, however, relax these conditions and allow facets that are surrounded by skew or non-coplanar edges or facets that self-intersect, have holes, or have no defined interiors (Gr\"{u}nbaum, 1994; Johnson, 2008). In fact, there is no universally agreed upon definition of a geometric polyhedron. The definition a work uses usually depends on the author's particular preferences, requirements, and objectives.

    While there is no consensus on what constitutes a geometric polyhedron, mathematicians generally agree on the conditions one must impose on its underlying vertex-edge-facet-cell incidence structure. This set of conditions defines a related mathematical object called an \textit{abstract polyhedron}. Essentially, it is a partially ordered set of elements called \textit{faces} that play analogous roles to the vertices, edges, and facets of its geometric counterpart. Since an abstract polyhedron is combinatorial in nature, it is devoid of metric properties and is best described by its group of \textit{automorphisms} or incidence-preserving face mappings.

    To lay out the foundation for a more rigorous treatment of geometric polyhedra, Johnson (2008) proposed the concept of a \textit{real polyhedron} using an abstract polyhedron as blueprint. In his theory, a real polyhedron is the \textit{realization} or the resulting figure when the faces of an abstract polyhedron are mapped to open sets in space. These associated open sets are selected so that they satisfy a set of conditions pertaining to their boundaries and intersections. Although Johnson's definition may not satisfy everyone's requirements, anchoring it to a well-accepted concept makes it less ambiguous and more consistent with existing notions and theories.

    In this work, we shall adopt a simplified version of Johnson's real polyhedron for the definition of a geometric polyhedron. Our main objective is to adapt a method based on \textit{Wythoff construction} (Coxeter, 1973) to generate a geometric polyhedron from a given abstract polyhedron $\mathcal{P}$ satisfying a \textit{regularity} property. The adapted method builds the figure by applying the image of an orthogonal representation of the automorphism group of $\mathcal{P}$ to a collection of open sets in space. The method is formulated and stated in a way that is amenable to algorithmic computations and suited for computer-based graphics generation. This work extends and further illustrates the ideas found in the work of Clancy (2005) and concretizes the algebraic version of Wythoff construction found in McMullen \& Schulte (2002).

    To illustrate the use of the method, we apply it to the abstract polyhedra whose automorphism groups are isomorphic to the non-crystallographic Coxeter group $H_3$ (Humphreys, 1992). The group has order 120 and can be described via the group presentation
    \[
        H_3 =
        \left\langle
        {
            \begin{tabular}{c|c}
                \multirow{2}{*}{$s_0, s_1, s_2$} & $s_0^2 = s_1^2 = s_2^2 = e$, \\
                & $(s_0s_1)^3 = (s_1s_2)^5 = (s_0s_2)^2 = e$ \\
            \end{tabular}
        }
        \right\rangle.
    \]
    Being the group of symmetries of icosahedral structures, $H_3$ has played a fundamental role in the study of mathematical models of quasicrystals (Chen \textit{et al.}, 1998; Patera \& Twarock, 2002), carbon onions, carbon nanotubes (Twarock, 2002), and viruses (Janner, 2006; Keef \& Twarock, 2009).

\section{Regular abstract polyhedra and string C-groups}
\label{sec:regularAbstractPolyhedraStringC-groups}

    We begin with a non-empty finite set $\mathcal{P}$ of elements called \textit{faces} that are partially ordered by a binary relation $\leq$. Two faces $F$, $G$ in $\mathcal{P}$ are said to be \textit{incident} if either $F \leq G$ or $G \leq F$. The incidence relations among the faces can be graphically represented using a \textit{Hasse diagram} in which a face is represented by a node and two nodes are connected by an edge if the corresponding faces are incident (Fig. \ref{fig:pentagram}(\textit{a})). Since a partial order is transitive, we shall omit edges corresponding to implied incidences.

    Given faces $F \leq G$, we define the \textit{section} $G/F$ of $\mathcal{P}$ to be the set of all faces $H$ incident to both $F$ and $G$, that is, $G/F = \{H \in \mathcal{P} \mid F \leq H \leq G\}$. Note that a section is also a partially ordered set under the same binary relation.

    A \textit{flag} $\Phi$ of length $i \geq -1$ is a totally ordered maximal subset $F_{-1} \leq F_0 \leq \cdots \leq F_i$ consisting of $(i + 1)$ faces of $\mathcal{P}$. Two flags $\Phi$, $\Psi$ are said to be \textit{adjacent} if they differ at exactly one face. Finally, $\mathcal{P}$ is said to be \textit{flag-connected} if for every pair of flags $\Phi$, $\Psi$, there is a finite sequence $\Phi = \Phi_0, \: \Phi_1, \ldots, \: \Phi_k = \Psi$ of adjacent flags.
% no need to distinguish/clarify notation of flag as a subset vs list of ordered faces?
% "joined by a sequence of adjacent flags" is vague because verb "joined" is vague, or no?

    \subsection{Abstract polyhedra}
    \label{subsec:abstractPolyhedra}

    For our purposes, we shall now restrict our treatment to partially ordered sets $\mathcal{P}$ that satisfy the following three properties:

    \begin{enumerate}
        \item[(P1)] $\mathcal{P}$ contains a unique \textit{least face} and a unique     \textit{greatest face}.
        \item[(P2)] Each flag of $\mathcal{P}$ has length 4 or contains exactly $5$ faces including the least face and the greatest face.
        \item[(P3)] $\mathcal{P}$ is \textit{strongly flag-connected}. That is, each section of $\mathcal{P}$ is flag-connected.
    \end{enumerate}

    Properties P1 and P2 imply that any face $F$ belongs to at least one flag and that the number of faces, excluding the least face, preceding it in any flag is constant. This constant, which we assign to be $-1$ for the least face, is called the \textit{rank} of $F$. We shall call a face of rank $i$ an \textit{$i$-face} and denote it by $F_i$ or $\Fij{i}{j}$ (with index $j$ for emphasis if there is more than one $i$-face). Thus, we denote the least face by $F_{-1}$ and the greatest face by $F_3$. When drawing a Hasse diagram, we shall adopt the convention of putting faces of the same rank at the same level and faces of different ranks at different levels arranged in ascending order of ranks.

    An \textit{abstract polyhedron} or a \textit{polytope of rank 3} (McMullen \& Schulte, 2002) is a partially ordered set $\mathcal{P}$ that satisfies properties P1, P2, P3 above, and property P4, also called the \textit{diamond property}, below:

    \begin{enumerate}
        \item[(P4)] If $F_{i - 1} \leq F_{i + 1}$, where $0 \leq i \leq 2$, then there are precisely two $i$-faces $F_i$ in $\mathcal{P}$ such that $F_{i - 1} \leq F_i \leq F_{i + 1}$.
    \end{enumerate}

    This definition of an abstract polyhedron is, in fact, a specific case of the more general definition of an \textit{abstract $n$-polytope} or \textit{polytope of rank $n$}. By rank of a polytope, we mean the rank of its greatest face. Borrowing terms from the theory of geometric polytopes, we shall refer to the $-1$-face of an abstract polyhedron as the empty face; a $0$-face as a \textit{vertex}; a $1$-face as an \textit{edge}; a $2$-face as a \textit{facet}; and the $3$-face as the \textit{cell}.

    \subsection{String C-groups}
    \label{subsec:stringC-groups}

    We can endow an abstract polyhedron $\mathcal{P}$ with an algebraic structure by defining a map on its faces that preserves both ranks and incidence relations. A bijective map $\gamma : \mathcal{P} \to \mathcal{P}$ is called an \textit{automorphism} if it is incidence-preserving on the faces:
    \[ F \leq G \text{ if and only if } \gamma(F) \leq \gamma(G). \]
    It is easy to verify using the properties of $\mathcal{P}$ that an automorphism is necessarily rank-preserving as well. By convention, we shall use the right action notation $F\gamma$ for the image $\gamma(F)$. Later, for nested mappings, it will be more convenient to use $\im{\gamma}{F}$ for this same image. We shall denote the group of all automorphisms of $\mathcal{P}$ by $\Gamma(\mathcal{P})$, or just $\Gamma$ when $\mathcal{P}$ is clear from context.

    An abstract polyhedron $\mathcal{P}$ is said to be \textit{regular} if $\Gamma$ acts transitively on its set of flags. Consequently, one can verify that the number $p$ of vertices incident to a facet and the number $q$ of facets incident to a vertex are both constant. These determine the \textit{(Schl\"{a}fli) type} $\{p, q\}$ of the regular polyhedron. Following the notation used in the \textit{Atlas of Small Regular Polytopes} (Hartley, 2006), we denote by $\{p, q\}^*m_x$ a regular polyhedron of type $\{p, q\}$ with automorphism group of order $m$. The index $x$, when present, distinguishes a polyhedron from other polyhedra of the same type with automorphism group of the same order.
% hard to imagine/arrive at "consequently, the number $p$ of vertices incident...", but may be too complicated to resolve without adding a paragraph or an example

    For a regular polyhedron of type $\{p, q\}$ , the automorphism group is a \textit{rank 3 string C-group of type $\{p, q\}$} and is best described as a pair $(\Gamma, \: T)$, which consists of a group $\Gamma$ and an ordered triple $T$ of distinct generating involutions $t_0, t_1, t_2$ that satisfy three properties:

    \begin{enumerate}
        \item string property: $t_0t_2 = t_2t_0$
        \item intersection property: $\langle{t_0, t_1}\rangle \cap \langle{t_1, t_2}\rangle = \langle{t_1}\rangle$
        \item order property: $\text{ord}(t_0t_1) = p$, $\text{ord}(t_1t_2) = q$
    \end{enumerate}

    Two string C-groups $(\Gamma, \: \{t_0, t_1, t_2\})$ and $(\Gamma', \: \{t_0', t_1', t_2'\})$ are considered equivalent if they have the same type and the map determined by $t_i \mapsto t_i'$ for $0 \leq i \leq 2$ is a group isomorphism. Since equivalence of string C-groups is dependent on the distinguished generating triples, we emphasize that two string C-groups may be considered distinct even if they are isomorphic as abstract groups.

    A fundamental result in the theory of abstract polytopes is the bijective correspondence between regular polyhedra and rank 3 string C-groups. It follows that the enumeration of regular polyhedra is equivalent to the enumeration of rank 3 string C-groups. Thus, given an arbitrary group $\Gamma$, one may determine all polyhedra with automorphism group isomorphic to $\Gamma$ by listing all generating triples $T$ of distinct involutions $t_0$, $t_1$, $t_2$ that satisfy the string and intersection conditions. For groups of relatively small order, it is straightforward to implement a listing procedure to accomplish this task in the software GAP (The GAP Group, 2019). We apply this procedure to the non-crystallographic Coxeter group $H_3$ and obtain 15 regular abstract \textit{$H_3$-polyhedra} with each belonging to one of 9 types summarized in Table \ref{tbl:H3-polyhedra}.
    
    \begin{table}[!htb]
    %\footnotesize
    \makebox[\textwidth][c]{
        \begin{tabular}{cccccc}
            \hline
            $\mathcal{P}$ & $t_0$ & $t_1$ & $t_2$ & $\dim{\mathcal{W}(\varphi_1, \; (H_3, T))}$ & $\dim{\mathcal{W}(\varphi_2, \; (H_3, T))}$ \\
            \hline
            %\rowcolor{lightyellow}
            $\{3, 5\}^*120$ & $s_0$ & $s_1$ & $s_2$ & $1$ & $1$ \\
            $\{3, 10\}^*120_a$ & $s_0$ & $s_1$ & $s_0s_2$ & $0$ & $0$ \\
            $\{3, 10\}^*120_b$ & $s_0s_2$ & $(s_1s_2)^2s_0s_1s_2s_1$ & $s_0$ & $0$ & $0$ \\
            %\rowcolor{lightyellow}
            $\{5, 3\}^*120$ & $s_2$ & $s_1$ & $s_0$ & $1$ & $1$ \\
            %\rowcolor{lightyellow}
            $\{5, 5\}^*120$ & $s_0$ & $s_1s_2s_1$ & $s_2$ & $1$ & $1$ \\
            $\{5, 6\}^*120_b$ & $s_0$ & $s_1s_2s_1$ & $s_0s_2$ & $0$ & $0$ \\
            $\{5, 6\}^*120_c$ & $s_0s_2$ & $s_1s_0s_2s_1$ & $s_2$ & $0$ & $0$ \\
            $\{5, 10\}^*120_a$ & $s_0$ & $s_1s_2s_1$ & $(s_1s_0s_2)^4s_1s_0$ & $0$ & $0$ \\
            $\{5, 10\}^*120_b$ & $s_0s_2$ & $s_1s_0s_2s_1$ & $s_0$ & $0$ & $0$ \\
            $\{6, 5\}^*120_b$ & $s_0$ & $(s_1s_0s_2)^3s_1$ & $s_0s_2$ & $0$ & $0$ \\
            %\rowcolor{lightyellow}
            $\{6, 5\}^*120_c$ & $s_0s_2$ & $s_1s_2s_1$ & $s_0$ & $1$ & $1$ \\
            %\rowcolor{lightyellow}
            $\{10, 3\}^*120_b$ & $s_0s_2$ & $s_1$ & $s_0$ & $1$ & $1$ \\
            $\{10, 3\}^*120_c$ & $s_0$ & $s_1s_0s_2s_1$ & $(s_1s_0s_2)^4s_1s_0$ & $0$ & $0$ \\
            $\{10, 5\}^*120_a$ & $s_0$ & $s_1s_0s_2s_1$ & $s_0s_2$ & $0$ & $0$ \\
            %\rowcolor{lightyellow}
            $\{10, 5\}^*120_b$ & $s_0s_2$ & $s_1$ & $s_2$ & $1$ & $1$ \\
            \hline
        \end{tabular}}
        \caption{The $H_3$-polyhedra with automorphism group generated by $T = \{t_0, t_1, t_2\}$. \label{tbl:H3-polyhedra}} %Those with non-zero Wythoff dimensions are highlighted.
    \end{table}

    \subsection{Coset-based construction method}
    \label{subsec:cosetConstructionMethod}

    Given a string C-group $(\Gamma, \: T)$, one may construct a regular abstract polyhedron $\mathcal{P}$ with automorphism group $\Gamma$. This is done by defining the cosets of certain subgroups of $\Gamma$ as the faces of $\mathcal{P}$ and partially ordering these cosets using a suitably chosen binary relation. In the theorem below, we employ the construction method in McMullen \& Schulte (2002).

    \begin{theorem}
        Suppose $(\Gamma, \: \{t_0, t_1, t_2\})$ is a string C-group of type $\{p, q\}$. Let $\Gamma_{-1} = \Gamma$, $\Gamma_3 = \Gamma$, and $\Gamma_i = \langle{t_k \mid k \neq i}\rangle$ for $0 \leq i \leq 2$. Then the following sequence of steps produces a regular abstract polyhedron $\mathcal{P}$ of type $\{p, q\}$ and automorphism group $\Gamma$:
        \begin{enumerate}
            \item Generate a complete list of right coset representatives $\gammaij{i}{j}$ of $\Gamma_i$ indexed by $1 \leq j \leq [\Gamma : \Gamma_i]$ for $-1 \leq i \leq 3$.

            \item Define $\mathcal{P}$ to be the set consisting of $F_{-1} = \Fij{-1}{1} = \Gamma_{-1}$, $F_3 = \Fij{3}{1} = \Gamma_3$, and $\Fij{i}{j} = \Gamma_i\gammaij{i}{j}$.

            \item Define a binary relation $\leq$ on $\mathcal{P}$ where $F_{i, \; j} \leq F_{i', \; j'}$ if and only if $i \leq i'$ and $\Gamma_i\gammaij{i}{j} \cap \Gamma_{i'}\gammaij{i'}{j'} \neq \varnothing$.
        \end{enumerate}
        Moreover, the number of $i$-faces of $\mathcal{P}$ is equal to the index of $\Gamma_i$ in $\Gamma$.
        \label{thm:cosetConstructionMethod}
    \end{theorem}

    As a consequence of this theorem, we may identify a regular polyhedron $\mathcal{P}$ with $\Gamma$ and an $i$-face $\Fij{i}{j}$ with a coset representative $\gammaij{i}{j}$ of $\Gamma_i$. For simplicity, we may assume this representative is the identity $e$ when $j = 1$.

    The Hasse diagram in Fig. \ref{fig:pentagram}(\textit{a}) is a section of the $H_3$-polyhedron $\{5, 3\}^*120$ in Table \ref{tbl:H3-polyhedra} consisting of a single empty face, $20$ vertices, $30$ edges, $12$ facets, and a single cell. This polyhedron results from applying Theorem \ref{thm:cosetConstructionMethod} to the string C-group $(H_3, \: \{s_2, s_1, s_0\})$.

\section{Regular geometric polyhedra and Wythoff construction}
\label{sec:geometricPolyhedra}
\sloppy

    Consider a regular abstract polyhedron $\mathcal{P}$ whose set of abstract $i$-faces is $\mathcal{P}_i$, where $-1 \leq i \leq 3$. Let $\Gamma$ be its automorphism group with distinguished generating triple $T = \{t_0, t_1, t_2\}$. By an open set of dimension $i$ in the Euclidean $n$-space $\mathbb{E}^n$, we mean a subset that is homeomorphic to an open set of $\mathbb{E}^i$.

    \subsection{Regular geometric polyhedra}
    \label{subsec:regularGeometricPolyhedra}

    Define the map $\rho_{-1} : \mathcal{P}_{-1} \to \mathbb{E}^n$ that sends the empty face $F_{-1}$ to the empty set $\mathcal{O}_{-1} = \varnothing$. Then for each $0 \leq i \leq 3$, recursively define a map $\rho_i : \mathcal{P}_i \to \mathbb{E}^n$ that sends each $i$-face $\Fij{i}{j}$ with index $1 \leq j \leq [\Gamma : \Gamma_i]$ to a non-empty open set $\Oij{i}{j}$ of dimension $i$. We require that the boundary of $\Oij{i}{j}$ be $\displaystyle{\bigcup_{0 \leq k < i}\left(\bigcup_{\Fij{k}{l} \leq \Fij{i}{j}}{\Oij{k}{l}}\right)}$, the union of the $\rho_k$-images of the lower rank $k$-faces $\Fij{k}{l}$ incident to $\Fij{i}{j}$.
    
\begin{figure}[!htb]
            \begin{tabular}{ccccc}
                \parbox[c]{1.3in}{\includegraphics[width = 1.3in]{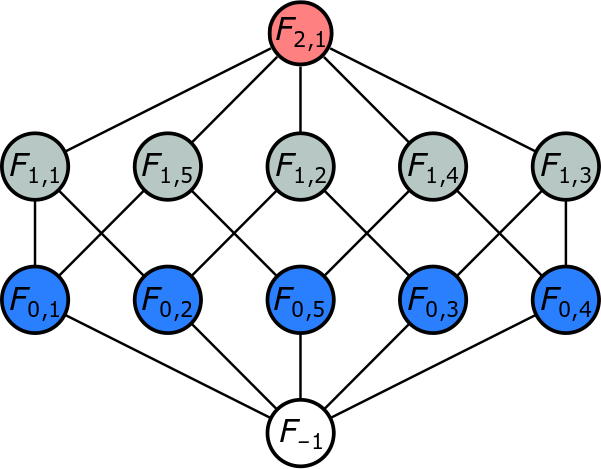}} 
                & 
                & \parbox[c]{1.3in}{\includegraphics[width = 1.3in]{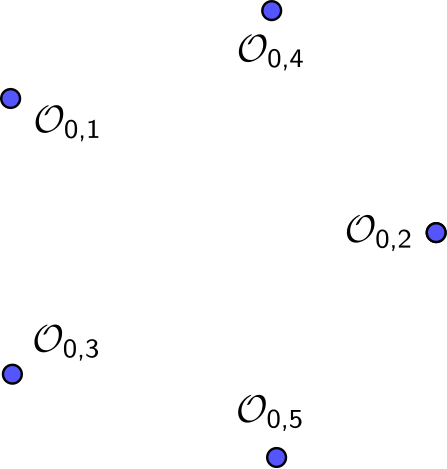}} 
                &
                & \parbox[c]{1.3in}{\includegraphics[width = 1.3in]{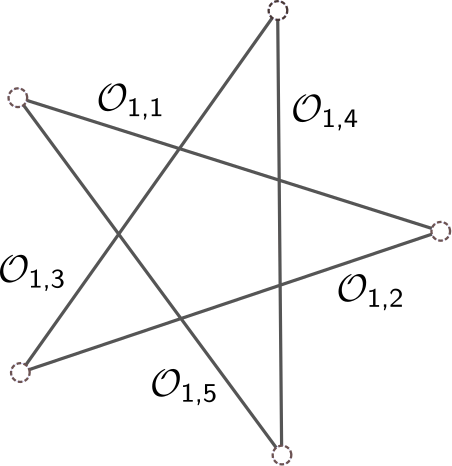}} \\
                (\textit{a}) 
                &
                & (\textit{b}) 
                &
                & (\textit{c}) \\
                \\[1mm]
                \parbox[c]{1.3in}{\includegraphics[width = 1.3in]{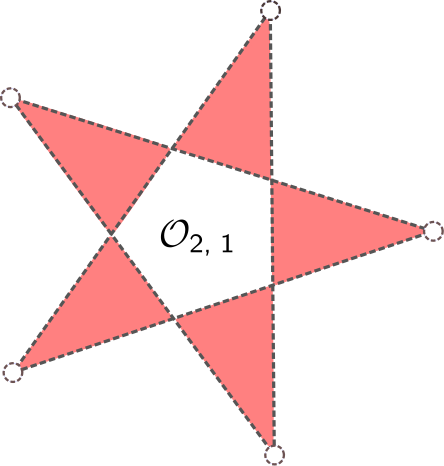}} 
                &
                & \parbox[c]{1.3in}{\includegraphics[width = 1.3in]{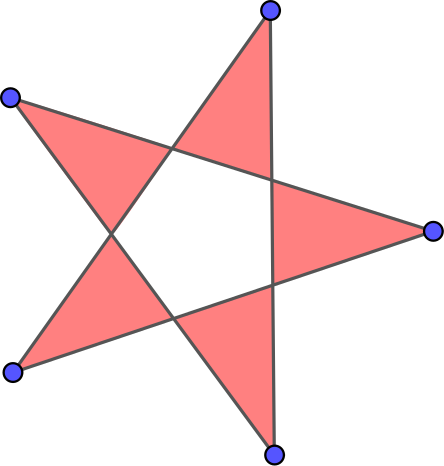}} 
                &
                & \parbox[c]{1.3in}{\includegraphics[width = 1.3in]{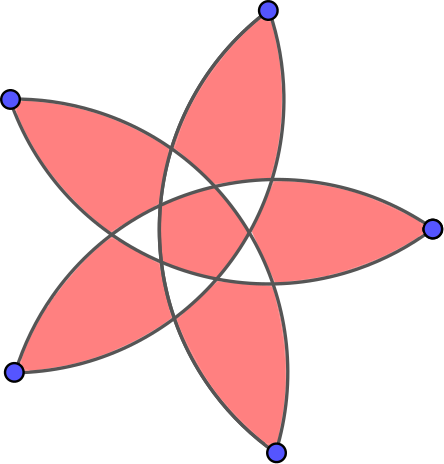}} \\
                (\textit{d}) 
                &
                & (\textit{e}) 
                &
                & (\textit{f})
            \end{tabular}
        \caption{(\textit{a}) Hasse diagram of a section of $\{5, 3\}^*120$. Geometric (\textit{b}) vertices, (\textit{c}) edges, (\textit{d}) facets corresponding to abstract faces that appear in the diagram. (\textit{e}) Regular pentagram obtained by combining the geometric faces in (\textit{a}) -- (\textit{d}). (\textit{f}) The regular pentagram in (e) with straight edges replaced by circular arcs. \label{fig:pentagram}}
\end{figure}

    \begin{illustration}
        We illustrate the images of the $i$-faces of $\{5, 3\}^*120$ that appear in the section represented by the Hasse diagram in Fig. \ref{fig:pentagram}(\textit{a}). These images partially determine maps $\rho_i$ for $0 \leq i \leq 2$. Take the points $\Oij{0}{j}$, $1 \leq j \leq 5$, in $\mathbb{E}^3$ (Fig. \ref{fig:pentagram}(\textit{b})) and let $\rho_0$ send each vertex $\Fij{0}{j}$ to $\Oij{0}{j}$; $\rho_1$ send each edge $\Fij{1}{j}$ to the open line segment $\Oij{1}{j}$ in Fig. \ref{fig:pentagram}(\textit{c}); and  $\rho_2$ send the facet $\Fij{2}{1}$ to the disconnected open region $\Oij{2}{1}$ in Fig. \ref{fig:pentagram}(\textit{d}). When these open sets of different dimensions are combined, we obtain the \textit{pentagram} shown in Fig. \ref{fig:pentagram}(\textit{e}). Choosing open arcs as the images of the edges instead and the disjoint union of suitably chosen open regions as the image of the lone facet, we obtain the figure illustrated in Fig. \ref{fig:pentagram}(\textit{f}). %Continuing recursively, we let $\rho_3$ be the map which sends the sole cell $\Fij{3}{1}$ to an open set $\Oij{3}{1}$ whose boundary contains this pentagram.
        \label{ill:pentagram}
    \end{illustration}

    The mapping $\rho : \mathcal{P} \to \mathbb{E}^n$ whose restriction to $\mathcal{P}_i$ is $\rho_i$ is called a \textit{geometric realization} of $\mathcal{P}$. To simplify the discussion, we limit ourselves to when $n = 3$, in which case $\rho$ is called a realization of \textit{full rank}. To distinguish between an $i$-face in $\mathcal{P}$ and its image under $\rho$, we call the former an \textit{abstract} $i$-face and the latter the \textit{realization} of this abstract $i$-face, or a \textit{geometric} $i$-face. Notice that the rank of an abstract face corresponds to the dimension of a geometric face in a realization. We now refer to the union of the geometric faces, which we denote by $\rho(\mathcal{P})$, as a \textit{regular geometric polyhedron} or, after identifying $\rho$ with its image, a \textit{geometric realization} of $\mathcal{P}$.

    We remark that the definition of a realization stated above is an interpretation of the standard definition (McMullen \& Schulte, 2002) in which abstract vertices are identified as points in space; edges as pairs of points; facets as sets of these pairs; and the cell as a collection of these sets of pairs. The standard definition, therefore, provides a blueprint to build a geometric polyhedron starting from its vertices and lets one exercise the freedom to choose Euclidean figures to represent abstract faces. Taking advantage of this freedom, we specify that abstract faces be associated to open sets with the appropriate dimension and boundary. This is to make the notion of a realization as wide-ranging as possible in order to cover typical figures representing known geometric polyhedra such as regular convex and star polyhedra. As we will see later, this will also allow one to generate polyhedra using curved edges and surfaces. Our definition of a realization is, in fact, consistent with the theory of \textit{real polytopes} formulated by Johnson (2008). Essentially, Johnson defines a realization to be an assembly of open regions in space with imposed restrictions pertaining to their boundaries and intersections.

    \subsection{Wythoff construction}
    \label{subsec:wythoffConstructionMethod}

    A \textit{faithful} realization $\rho$ is one where each induced map $\rho_i$ is injective. That is, distinct abstract $i$-faces $\Fij{i}{j}$ are sent to distinct geometric $i$-faces $\Oij{i}{j}$. It follows that there is a bijective correspondence between the set of $\Fij{i}{j}$'s and the set of $\Oij{i}{j}$'s that preserves ranks and incidence relations in the former, and dimensions and boundary relations in the latter.

    A \textit{symmetric} realization, on the other hand, is one where each automorphism $\gamma \in \Gamma$ corresponds to an isometry of $\mathbb{E}^3$ that symmetrically permutes the $\Oij{i}{j}$'s. More specifically, a symmetric realization presupposes the existence of an orthogonal representation $\varphi : \Gamma \to O(3)$ that satisfies
    \begin{equation}
        \rho_i(\im{\gamma}{\Fij{i}{j}}) = \im{\varphi(\gamma)}{\rho_i(\Fij{i}{j})} = \im{\varphi(\gamma)}{\Oij{i}{j}}.
        \label{eq:symmetric}
    \end{equation}
    We recall that $\varphi(\gamma)$ acts on $\mathbb{E}^3$ and preserves the usual Euclidean inner product. Consequently, for a fixed orthogonal basis, we may represent each $\gamma$ with a $3 \times 3$ real orthogonal matrix. We denote the image $\varphi(\Gamma)$ of this representation by $G(\rho(\mathcal{P}))$, or just $G$ when $\rho(\mathcal{P})$ is clear from context. We remark that $G$ is the symmetry group of the geometric polyhedron whenever $\rho$ itself is faithful and symmetric. Such a realization always implies that $\varphi$ is faithful:

    \begin{proposition}
        Let $\rho$ be a faithful symmetric realization of $\mathcal{P}$. If $\varphi : \Gamma \to O(3)$ is the associated orthogonal representation, then $\varphi$ is faithful.
        \label{prop:faithfulRep}
    \end{proposition}

    \begin{proof}
        It suffices to show that if $\varphi(\gamma)$ is the identity isometry $\iota$, then $\gamma$ is the identity automorphism $e$. By equation \eqref{eq:symmetric}, we have
        \[ \rho_i(\im{\gamma}{\Fij{i}{j}}) = \im{\varphi(\gamma)}{\rho_i(\Fij{i}{j})} = \im{\iota}{\rho_i(\Fij{i}{j})} = \rho_i(\Fij{i}{j}) \]
        for any abstract $i$-face $\Fij{i}{j}$. Thus, $\rho_i(\text{Im}(\gamma, \: \Fij{i}{j})) = \rho_i(\Fij{i}{j})$, which is equivalent to $\text{Im}(\gamma, \: \Fij{i}{j}) = \Fij{i}{j}$ by faithfulness of $\rho$. Since $\Fij{i}{j}$ is arbitrary, $\gamma$ must be $e$. Consequently, $\varphi$ is faithful.
    \end{proof}

    From this point forward, we restrict ourselves to realizations $\rho$ which are both faithful and symmetric. With these properties not only do we have a correspondence between abstract and geometric faces, we also have a correspondence between the action of the automorphism group on the abstract faces and the action of the symmetry group on the corresponding geometric faces. Consequently, any geometric polyhedron obtained from $\rho$ will automatically satisfy regularity or transitivity of geometric flags. Thus, to construct $\rho$, we must employ a faithful orthogonal representation by Proposition \ref{prop:faithfulRep}. The group $H_3$ has two such irreducible representations (Koca \& Koca, 1998):
    \begin{equation*}
        \varphi_1 \: : \: s_0 \mapsto
        \begin{bmatrix}
            -1 & 0 & 0 \\
            0 & 1 & 0 \\
            0 & 0 & 1 \\
        \end{bmatrix}, \;\;
        s_1 \mapsto
        \frac{1}{2}
        \begin{bmatrix}
            1 & -\tau & -\sigma \\
            -\tau & \sigma & 1 \\
            -\sigma & 1 & \tau \\
        \end{bmatrix}, \;\;
        s_2 \mapsto
        \begin{bmatrix}
            1 & 0 & 0 \\
            0 & -1 & 0 \\
            0 & 0 & 1 \\
        \end{bmatrix},
    \end{equation*}

    \begin{equation*}
        \varphi_2 \: : \: s_0 \mapsto
        \begin{bmatrix}
            -1 & 0 & 0 \\
            0 & 1 & 0 \\
            0 & 0 & 1 \\
        \end{bmatrix}, \;\;
        s_1 \mapsto
        \frac{1}{2}
        \begin{bmatrix}
            1 & -\sigma & -\tau \\
            -\sigma & \tau & 1 \\
            -\tau & 1 & \sigma \\
        \end{bmatrix}, \;\;
        s_1 \mapsto
        \begin{bmatrix}
            1 & 0 & 0 \\
            0 & -1 & 0 \\
            0 & 0 & 1 \\
        \end{bmatrix},
    \end{equation*}
    where $\tau = \frac{1 + \sqrt{5}}{2}$ and $\sigma = \frac{1 - \sqrt{5}}{2}$.

    We now describe an explicit construction method in Theorem \ref{thm:WythoffConstruction} to obtain a realization of a polyhedron from a string C-group $(\Gamma, \: T)$. Recall earlier that we may identify an $i$-face $\Fij{i}{j}$ with a coset representative $\gammaij{i}{j}$ of $\Gamma_i$.
    \begin{theorem}
        Let $(\Gamma, \: T)$ be a string C-group which characterizes the automorphism group of a regular abstract polyhedron $\mathcal{P}$ and let $\varphi$ be a faithful irreducible orthogonal representation of $\Gamma$. Then the following sequence of steps produces a faithful symmetric realization $\rho$ of $\mathcal{P}$:
        \begin{enumerate}
            \item Generate a complete list of right coset representatives $\gammaij{i}{j}$ of $\Gamma_i$ with index $1 \leq j \leq [\Gamma : \Gamma_i]$ for $0 \leq i \leq 3$.
            \item Compute the matrix representations $\varphi(\gammaij{i}{j})$ of the coset representatives $\gammaij{i}{j}$.
            \item Compute the \textit{Wythoff space}
                \[ \mathcal{W}(\varphi, \; (\Gamma, \: T)) = \{\mathbf{x} \in \mathbb{E}^3 \mid \im{\varphi(t_1)}{\mathbf{x}} = \im{\varphi(t_2)}{\mathbf{x}} = \mathbf{x}\} \]
                associated with the pair $(\varphi, \; (\Gamma, \: T))$. This space consists of points in $\mathbb{E}^3$ that are fixed by both $\varphi(t_1)$ and $\varphi(t_2)$.
            \item Pick a point $\mathbf{x} \in \mathcal{W}(\varphi, \; (\Gamma, \: T))$ and let $\Oij{0}{1}$ be $\mathbf{x}$.
            \item For $1 \leq i \leq 3$:
                \begin{enumerate}
                    \item Determine the abstract $(i - 1)$-faces incident to the base abstract $i$-face. Equivalently, determine the indexing set $J_i = \{j \mid \Gamma_{i - 1}\gammaij{i - 1}{j} \cap \Gamma_i \neq \varnothing\}$.
                    \item Compute the open sets $\Oij{i - 1}{j} = \im{\varphi(\gammaij{i - 1}{j})}{\Oij{i - 1}{1}}$ for each $j \in J_i$.
                    \item Let $\Oij{i}{1}$ be an open set that is bounded by $\Oij{i - 1}{j}$ for $j \in J_i$ and stabilized in $G$ by $G_i = \varphi(\Gamma_i)$.
                \end{enumerate}
            \item For $0 \leq i \leq 3$, define $\rho_i$ to be the map on $\mathcal{P}_i$ that sends each $\Fij{i}{j}$ to $\Oij{i}{j} = \im{\varphi(\gammaij{i}{j})}{\Oij{i}{1}}$ for $1 \leq j \leq [\Gamma : \Gamma_i]$.
            \item Define $\rho$ to be the map on $\mathcal{P}$ whose restriction to $\mathcal{P}_i$ is $\rho_i$.
        \end{enumerate}
        \label{thm:WythoffConstruction}
    \end{theorem}

    \begin{proof}
        To prove the theorem, we need only show that $\rho$ is faithful and symmetric. To this end, let $\gamma \in \Gamma$ and $\Fij{i}{j}, \Fij{i}{k} \in \mathcal{P}_i$.

        Suppose that $\rho_i(\Fij{i}{j}) = \rho_i(\Fij{i}{k})$. By the definition of $\Oij{i}{j}$ in Step 6, we obtain
        \[ \rho_i(\Fij{i}{j}) = \im{\varphi(\gammaij{i}{j})}{\Oij{i}{1}} \text{ and } \rho_i(\Fij{i}{k}) = \im{\varphi(\gammaij{i}{k})}{\Oij{i}{1}}, \]
        which implies that $\varphi(\gammaij{i}{j}\gammaij{i}{k}^{-1})$ stabilizes $\Oij{i}{1}$. Since $\Oij{i}{1}$ is chosen so that its stabilizer in $G$ is $G_i$, we must have $\gammaij{i}{j}\gammaij{i}{k}^{-1} \in \Gamma_i$. Thus, $\Gamma_i\gammaij{i}{j} = \Gamma_i\gammaij{i}{k}$, or equivalently, $\Fij{i}{j} = \Fij{i}{k}$. Hence, $\rho$ is faithful.

        To show that $\rho$ is symmetric as well, let $\im{\gamma}{\Fij{i}{j}} = \Fij{i}{k}$. It follows that $(\Gamma_i\gammaij{i}{j})\gamma = \Gamma_i\gammaij{i}{k}$ and so, $\gamma = \gammaij{i}{j}^{-1}\sigma\gammaij{i}{k}$ for some $\sigma \in \Gamma_i$. The image of  $\Oij{i}{j}$ under $\varphi(\gamma)$ is
        \begin{align*}
            \im{\varphi(\gamma)}{\Oij{i}{j}}
            &= \im{\varphi(\gammaij{i}{j}^{-1}\sigma\gammaij{i}{k})}{\Oij{i}{j}} \\
            &= \im{\varphi(\gammaij{i}{j}^{-1})\varphi(\sigma)\varphi(\gammaij{i}{k})}{\Oij{i}{j}} \\
            &= \im{\varphi(\sigma)\varphi(\gammaij{i}{k})}{\Oij{i}{1}} \\
            &= \im{\varphi(\gammaij{i}{k})}{\Oij{i}{1}}
        \end{align*}
        where each component of $\varphi(\gammaij{i}{j}^{-1})\varphi(\sigma)\varphi(\gammaij{i}{k})$ is sequentially applied to $\Oij{i}{j}$ from left to right to conform with the right action of $\Gamma$ on $\mathcal{P}_i$. We then have
        \[ \rho_i(\im{\gamma}{\Fij{i}{j}}) = \rho_i(\Fij{i}{k}) = \Oij{i}{k} = \im{\varphi(\gammaij{i}{k})}{ \Oij{i}{1}} = \im{\varphi(\gamma)}{\Oij{i}{j}}. \]
        Hence, $\rho$ is symmetric.
    \end{proof}

    The procedure described in Theorem \ref{thm:WythoffConstruction} is an algebraic version of the method of \textit{Wythoff construction} named after the Dutch mathematician Willem Abraham Wythoff (McMullen \& Schulte, 2002). Wythoff's original geometric version is used to construct uniform tessellations. It relies on a kaleidoscope-like setup in which three reflection mirrors bound what becomes a fundamental triangle of the resulting uniform figure (Coxeter, 1973). In Theorem \ref{thm:WythoffConstruction}, the fixed spaces of the generators in $T$, which may not necessarily be reflections, play the role of the mirrors.

    For a string C-group of type $\{p, q\}$, we may compute the dimension of the Wythoff space using the formula (Clancy, 2005)
    \begin{equation}
        \dim{\mathcal{W}(\varphi, \; (\Gamma, \: T))} = \frac{1}{2q}\sum_{\gamma \in \langle{t_1, t_2}\rangle}{\text{Tr } \varphi(\gamma)},
        \label{eq:WythoffDimension}
    \end{equation}
    where $\text{Tr } \varphi(\gamma)$ denotes the trace of $\varphi(\gamma)$. We note that if $\dim{\mathcal{W}(\varphi, \; (\Gamma, \: T))} = 0$, we do not obtain any realization via Theorem \ref{thm:WythoffConstruction}. If $\dim{\mathcal{W}(\varphi, \; (\Gamma, \: T))} = 1$, on the other hand, any two choices for the base geometric vertex will just be scalar multiples of each other. It follows that a 1-dimensional Wythoff space produces only \textit{algebraically equivalent} realizations. Different choices for the open image of a base face, however, may yield polyhedra that are topologically different.

        \begin{figure}
            \begin{tabular}{ccc}
                \parbox[c]{1.7in}{\includegraphics[width = 1.7in]{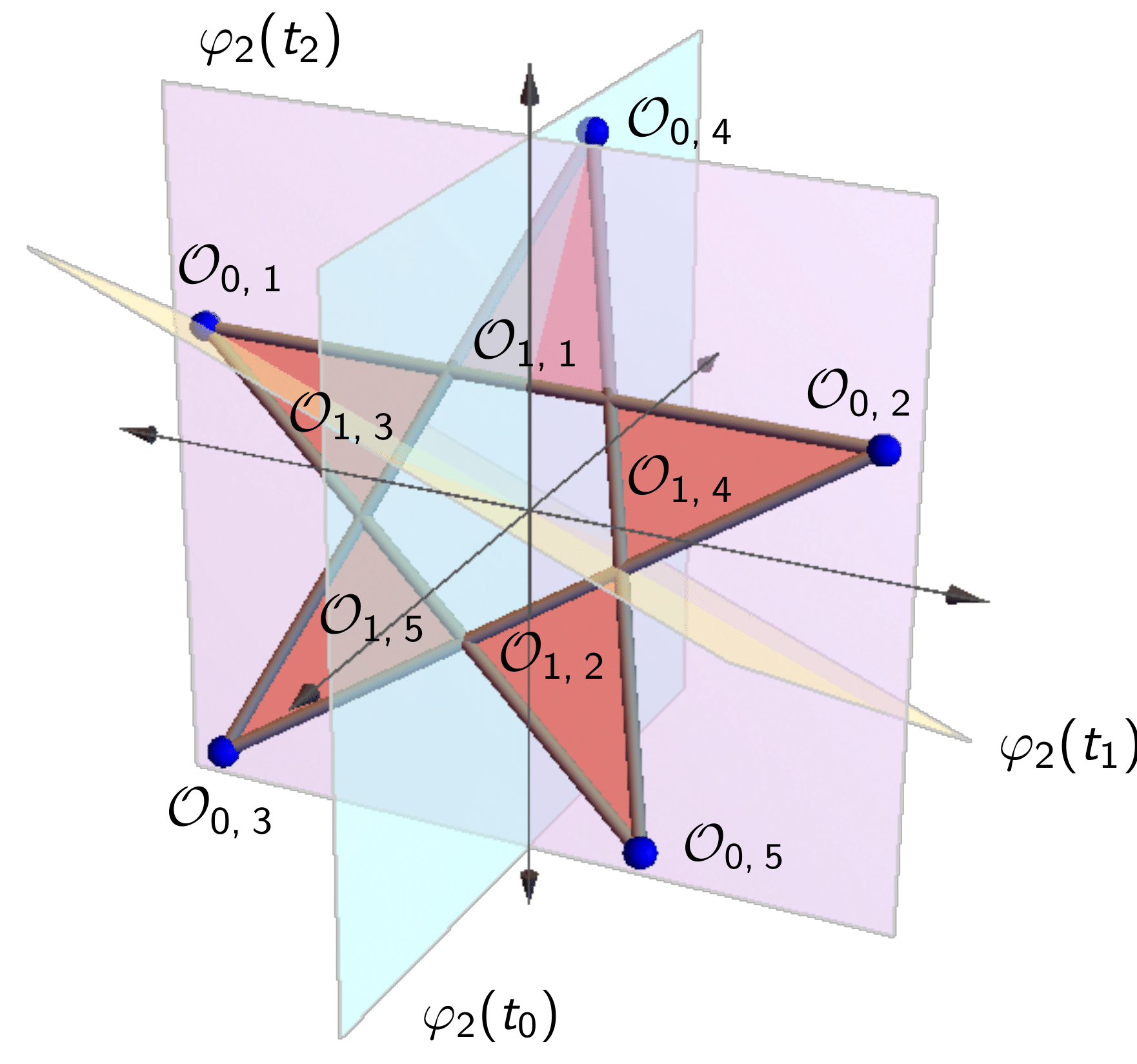}} & & \parbox[c]{1.7in}{\includegraphics[width = 1.3in]{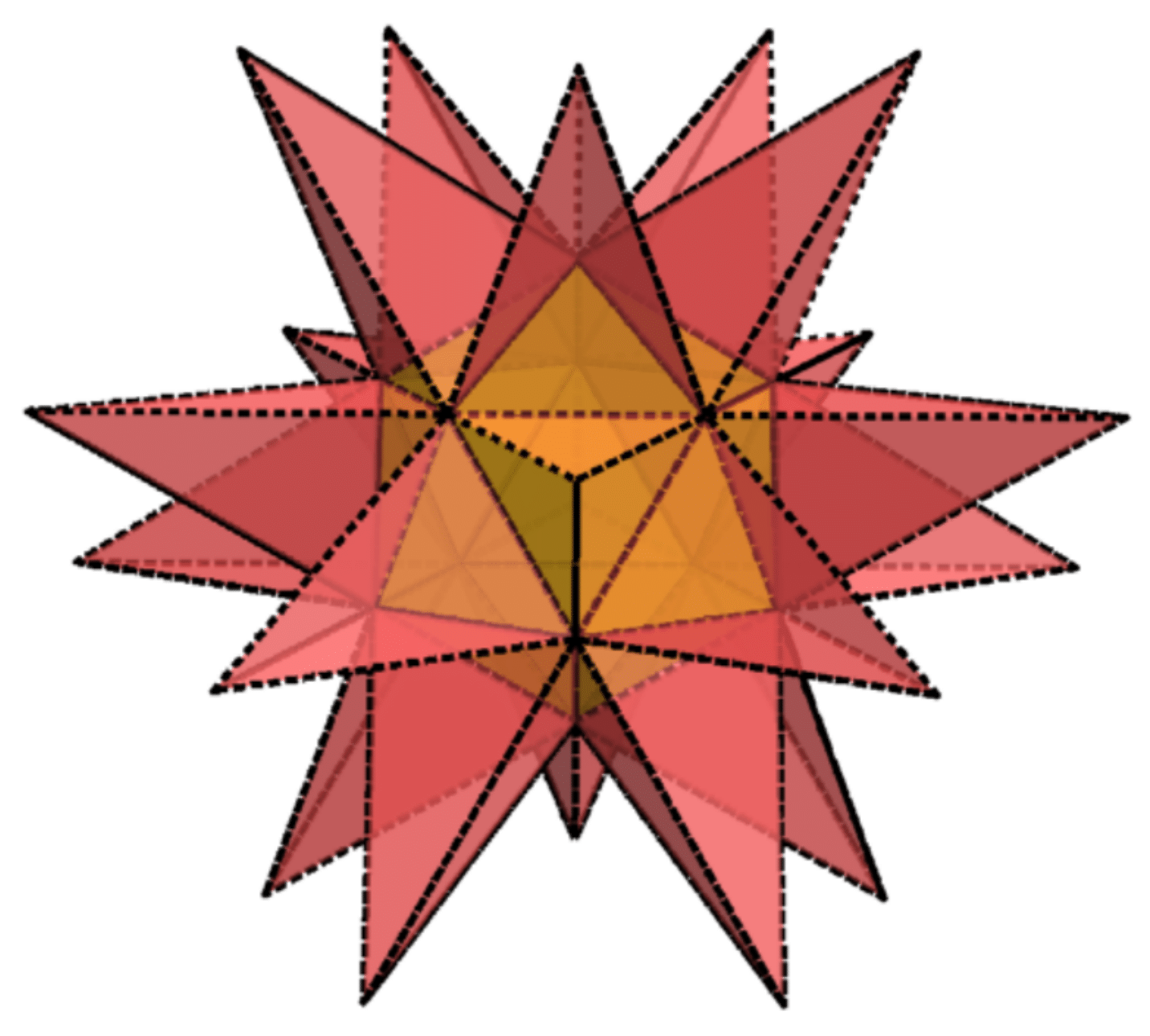}} \\
                (\textit{a}) & & (\textit{b}) \\
                \\[1mm]
                \parbox[c]{1.7in}{\includegraphics[width = 1.7in]{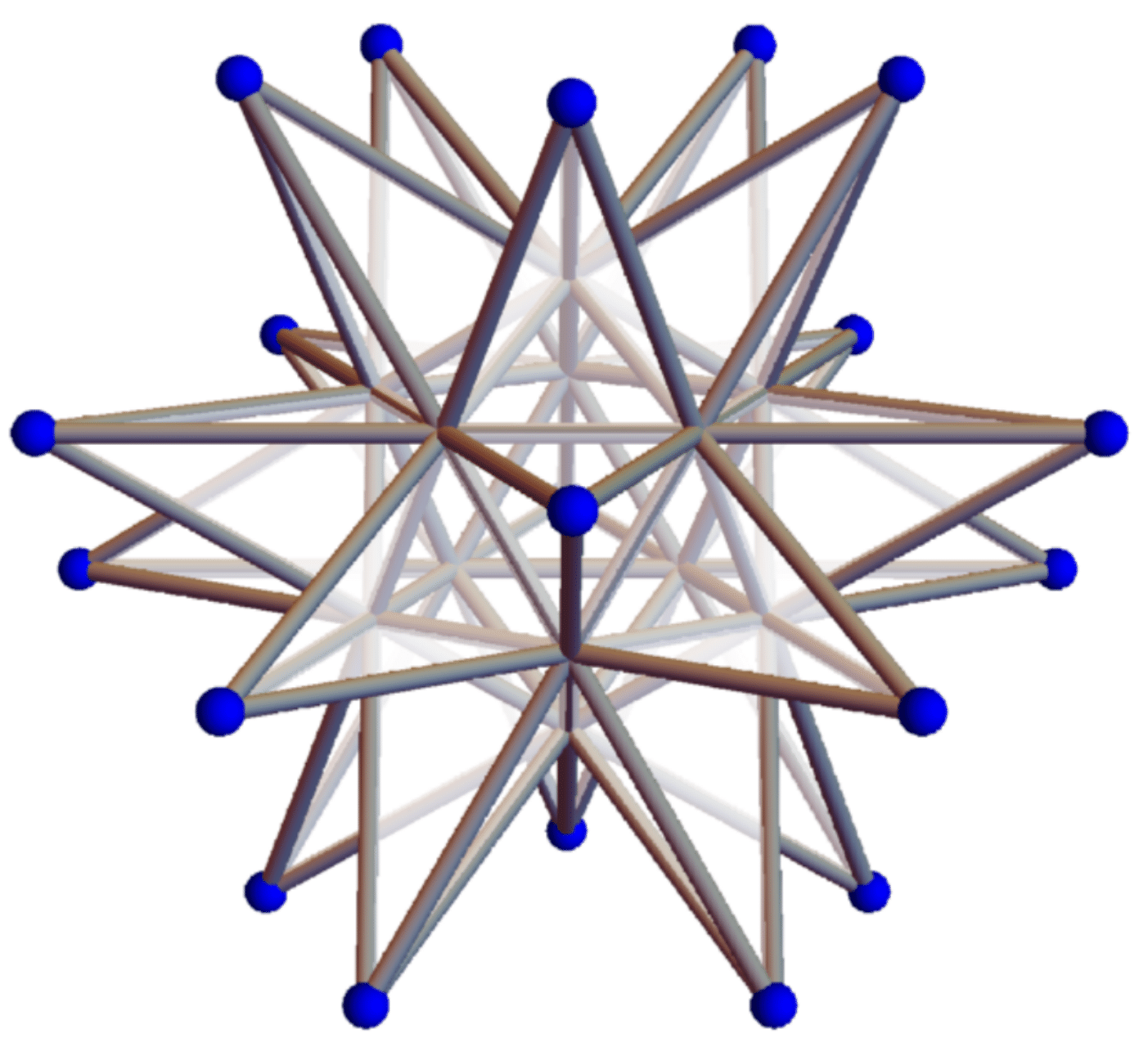}} & & \parbox[c]{1.7in}{\includegraphics[width = 1.7in]{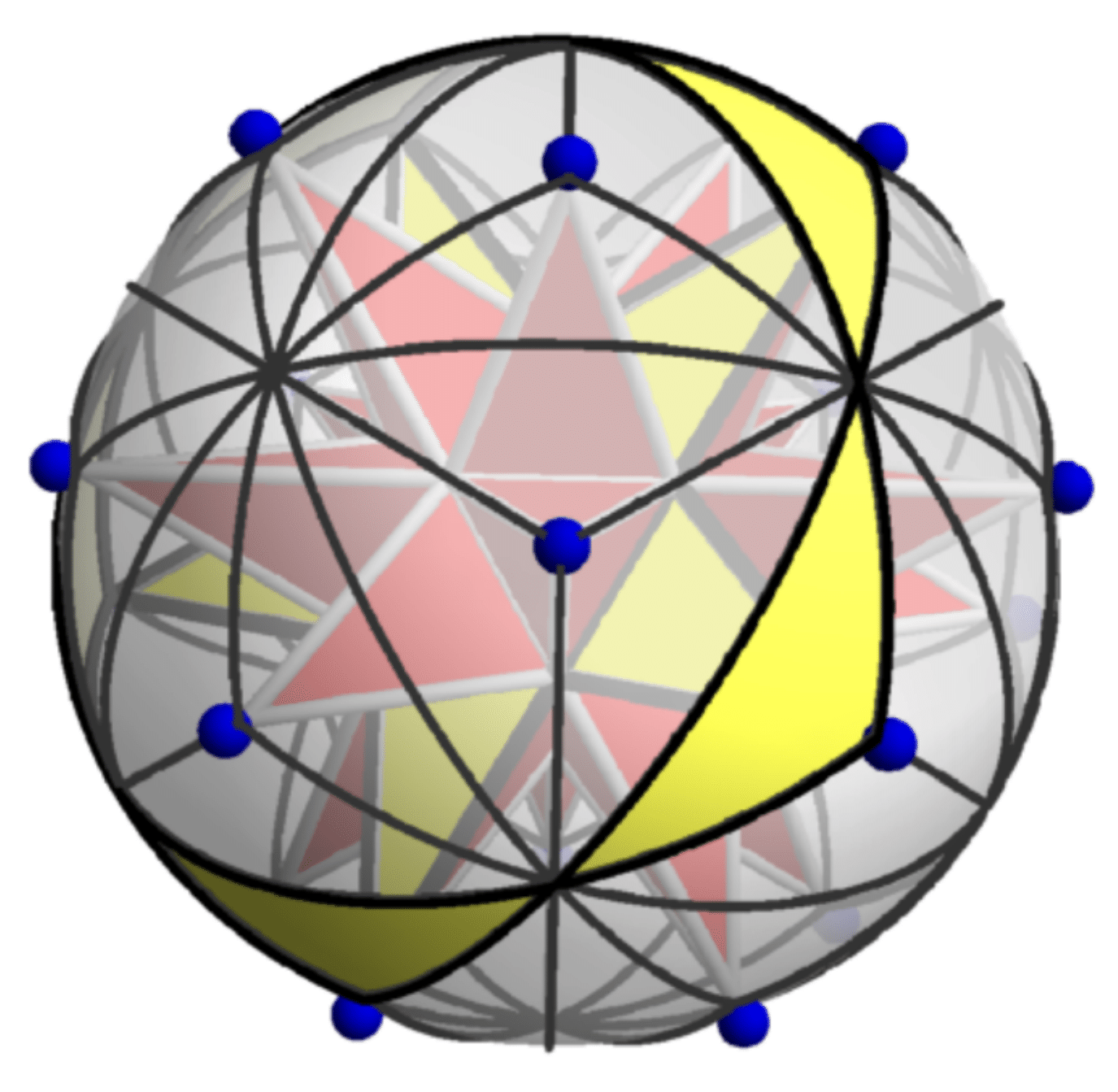}} \\
                (\textit{c}) & & (\textit{d}) \\
            \end{tabular}
            \caption{(\textit{a}) Base geometric vertex, edge, facet of $\rho_{\text{st}}(\{5, 3\}^*120)$. (\textit{b}) Base geometric cell of $\rho_{\text{st}}(\{5, 3\}^*120)$ with icosahedral hole. (\textit{c}) Union of the geometric vertices, edges of $\rho_{\text{st}}(\{5, 3\}^*120)$. (\textit{d}) Spherical realization $\rho_{\text{sp}}(\{5, 3\}^*120)$ circumscribing the star realization $\rho_{\text{st}}(\{5, 3\}^*120)$. \label{fig:gsDodecahedron}}
        \end{figure}

    \begin{illustration}
        We now illustrate the use of Theorem \ref{thm:WythoffConstruction} to create a realization $\rho_{\text{st}}$ of $\{5, 3\}^*120$ with automorphism group $\Gamma = H_3$ generated by the triple $T$ consisting of $t_0 = s_2$, $t_1 = s_1$, $t_2 = s_0$. Employing the representation $\varphi_2$, we have the following generating matrices for $G = \varphi_2(\Gamma)$:
        \[
            \varphi_2(t_0) =
            \begin{bmatrix}
                1 & 0 & 0 \\
                0 & -1 & 0 \\
                0 & 0 & 1 \\
            \end{bmatrix}, \;\;
            \varphi_2(t_1) =
            \frac{1}{2}
            \begin{bmatrix}
                1 & -\sigma & -\tau \\
                -\sigma & \tau & 1 \\
                -\tau & 1 & \sigma \\
            \end{bmatrix}, \;\;
            \varphi_2(t_2) =
            \begin{bmatrix}
                -1 & 0 & 0 \\
                0 & 1 & 0 \\
                0 & 0 & 1 \\
            \end{bmatrix}
        \]
        These three generators correspond to reflections of $\mathbb{E}^3$ with the first and third having the $xz$-plane and $yz$-plane, respectively, as mirrors.

        For each $0 \leq i \leq 3$, we use GAP to generate a complete list of right coset representatives $\gammaij{i}{j}$ of $\Gamma_i$, where $1 \leq j \leq [\Gamma : \Gamma_i]$, and their corresponding matrix representations $\varphi_2(\gammaij{i}{j})$.

        By formula \eqref{eq:WythoffDimension}, we obtain $\dim{\mathcal{W}(\varphi, \; (\Gamma, \: T))} = 1$. We compute the Wythoff space by finding a basis for the intersection of the 1-eigenspaces of $\varphi_2(t_1)$ and $\varphi_2(t_2)$. Using the Zassenhaus algorithm yields $\mathcal{W}(\varphi_2, \; (\Gamma, \; T)) = \text{span} \{(0, 1, 1 + \sigma)\} \subseteq \mathbb{E}^3$.

        As explained earlier, we construct the base geometric $i$-face $\Oij{i}{1}$ for $1 \leq i \leq 3$ taking into account not only $\Oij{i - 1}{1}$, but also the realizations $\Oij{i - 1}{j}$ of the $(i - 1)$-faces $\gammaij{i - 1}{j}$ incident to $\gammaij{i}{1}$. This ensures that, at each stage, $\Oij{i}{1}$ is bounded by these $\Oij{i - 1}{j}$'s as required by the definition of a realization. In addition, $\Oij{i}{1}$ must be chosen carefully so that its stabilizer is $G_i$.
        
        \begin{itemize}
            \item Base geometric vertex: Pick the point $(0, 1, 1 + \sigma)$ in the Wythoff space and let this be $\Oij{0}{1}$.
            \item Base geometric edge: Aside from $\gammaij{0}{1} = e$, only the vertex $\gammaij{0}{2} = t_0$ is incident to the base edge $\gammaij{1}{1} = e$. We define $\Oij{1}{1}$ to be the open line segment (Fig. \ref{fig:gsDodecahedron}(\textit{a})) whose endpoints are $\Oij{0}{1}$ and $\Oij{0}{2} = \im{\varphi_2(t_0)}{\Oij{0}{1}} = (0, -1, 1 + \sigma)$. This segment is stabilized by $G_1$ with $\varphi_2(t_0)$ interchanging these endpoints and $\varphi_2(t_2)$ fixing them.
            \item Base geometric facet: There are 5 edges incident to the base facet $\gammaij{2}{1} = e$. These are $\gammaij{1}{1} = e$, $\gammaij{1}{2} = t_0t_1$, $\gammaij{1}{3} = (t_0t_1)^2$, $\gammaij{1}{4} = t_1t_0t_1$, $\gammaij{1}{5} = t_0t_1$. We define $\Oij{2}{1}$ to be the open regular pentagram (Fig. \ref{fig:gsDodecahedron}(\textit{a})) bounded by the segments $\Oij{1}{j} = \im{\varphi_2(\gammaij{1}{j})}{\Oij{1}{1}}$ for $1 \leq j \leq 5$ with endpoints $\Oij{0}{1}$, $\Oij{0}{2}$, $\Oij{0}{3} = (\sigma, -\sigma, \sigma)$, $\Oij{0}{4} = (1 + \sigma, 0, 1)$, $\Oij{0}{5} = (\sigma, \sigma, \sigma)$ as shown in the figure. It is straightforward to verify that this pentagram is stabilized by $G_2$ with $\varphi_2(t_0)$ fixing $\Oij{1}{1}$, $\varphi_2(t_1)$ fixing $\Oij{1}{3}$, and either permuting the remaining segments.
            \item Base geometric cell: There are 12 facets incident to the base cell $\gammaij{3}{1} = e$. These are $\gammaij{2}{1} = e$, $\gammaij{2}{2} = t_2$, $\gammaij{2}{3} = t_1t_2$, $\gammaij{2}{4} = t_0t_1t_2$, $\gammaij{2}{5} = t_1t_0t_1t_2$, $\gammaij{2}{6} = t_2t_1t_0t_1t_2$, $\gammaij{2}{7} = (t_0t_1)^2t_2$, $\gammaij{2}{8} = t_0t_2t_1t_0t_1t_2$, $\gammaij{2}{9} = t_1t_0t_2t_1t_0t_1t_2$, $\gammaij{2}{10} = t_0t_1t_0t_2t_1t_0t_1t_2$, $\gammaij{2}{11} = (t_1t_0)^2t_2t_1t_0t_1t_2$, $\gammaij{2}{12} = t_2(t_1t_0)^2t_2t_1t_0t_1t_2$. We define $\mathcal{O}_3$ to be the open region (Fig. \ref{fig:gsDodecahedron}(\textit{b})) bounded by the open pentagrams $\Oij{2}{j} = \im{\varphi_2(\gammaij{2}{j})}{\Oij{2}{1}}$ for $1 \leq j \leq 12$. The region $\mathcal{O}_3$ is the disjoint union of 20 open triangular pyramids whose bases form the bounding surface of a regular \textit{icosahedron}. We can thus informally describe $\mathcal{O}_3$ as an open ``spiky'' solid with an icosahedral hole at its core. It will follow that $\mathcal{O}_3$ is stabilized by $G$ after verifying that each generator of $G$ either fixes a bounding pentagram or sends it to another one.
        \end{itemize}

        The resulting geometric polyhedron $\rho_{\text{st}}(\{5, 3\}^*120)$ is obtained by getting the union of the geometric vertices, edges in Fig. \ref{fig:gsDodecahedron}(\textit{c}) and the geometric facets, cell in Fig. \ref{fig:gsDodecahedron}(\textit{b}).
        \label{ill:realization53}
    \end{illustration}

    \subsection{Geometric faces}
    \label{subsec:geometricFaces}

    Here we describe four different families of realizations -- \textit{spherical}, \textit{convex}, \textit{star}, \textit{skew} -- classified according to the geometry and relative arrangements of their associated open sets. These were chosen to demonstrate the capability of Theorem \ref{thm:WythoffConstruction} to later produce a realization for each of the regular $H_3$-polyhedra in Table \ref{tbl:H3-polyhedra}. It is important to note that other families of open sets may also be chosen and the four enumerated here are by no means the only options available.

    \subsubsection{Spherical realization}
    \label{subsubsec:spherical}

    Since orthogonal matrices are isometric, a sphere is a natural space for a geometric polyhedron to inhabit. For a \textit{spherical realization} denoted by $\rho_{\text{sp}}$, we define the base geometric vertex as a point on the surface of a fixed sphere; the base geometric edge as an open spherical arc; the base geometric facet as an open spherical polygon; and the geometric cell as the sphere's interior. Observe that the geometric faces  excluding the cell tile the surface of the sphere. Thus, we may regard a spherical realization as a covering of the surface of a sphere by spherical polygons.

    \subsubsection{Convex and star realizations}
    \label{subsubsec:convexStar}

    Suppose that, in a spherical realization, we set the base geometric edge to be an open line segment instead of a spherical arc. Provided that the resulting bounding edges of the base geometric facet are coplanar, we may define a \textit{classical realization} which is either \textit{convex} and denoted by $\rho_{\text{co}}$ or \textit{star} and denoted by $\rho_{\text{st}}$. If a pair of edges (resp. facets) intersect, we set the base geometric facet (resp. cell) to be the union of disconnected open regions bounded by its incident edges (resp. facets). Otherwise, we define it as the interior of the convex hull of these edges (resp. facets).

    The presence of intersecting edges or facets characterizes a star realization. That is, the resulting star polyhedron is \textit{polymorphic} and has a cell which generally consists of the union of two or more distinct open regions in space (Johnson, 2008).

    The convexity of the resulting geometric polyhedron, on the other hand, characterizes a convex realization. That is, a convex polyhedron is a solid where each geometric $i$-face is the interior of the convex hull of its bounding geometric $(i - 1)$-faces.
    
    \begin{figure}[!htb]
        \begin{tabular}{ccc}
            \parbox[c]{1.7in}{\includegraphics[width = 1.7in]{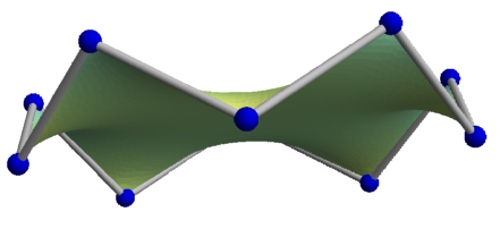}} & \parbox[c]{1.7in}{\includegraphics[width = 1.7in]{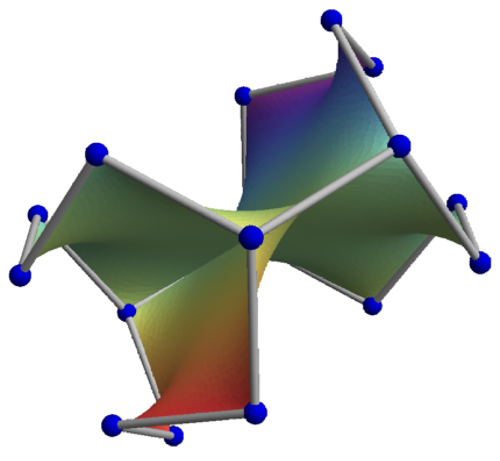}} & \parbox[c]{1.7in}{\includegraphics[width = 1.7in]{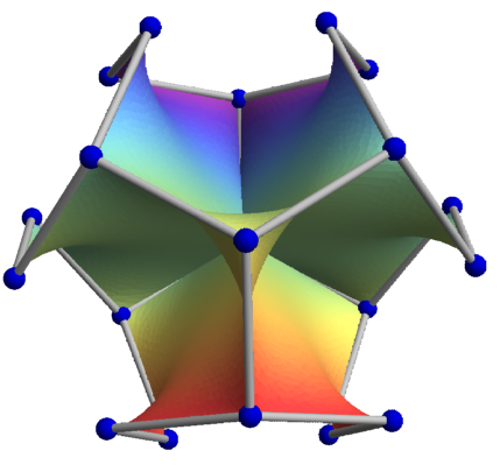}} \\
            (\textit{a}) & (\textit{b}) & (\textit{c}) \\
        \end{tabular}
        \caption{(\textit{a}) Base geometric facet of $\rho_{\text{sk}}(\{10, 3\}^*120_b)$ with (\textit{b}) one and (\textit{c}) two of its symmetric copies, bounding a region in space with non-zero volume. \label{fig:103skewFacets}}
    \end{figure}
	
    \subsubsection{Skew realization}
    \label{subsubsec:skew}

    Consider the scenario in which the geometric edges are open line segments as in a convex or a star realization, but the resulting bounding edges of the base geometric facet are non-coplanar. In this case, we set the base geometric facet to be the interior of the \textit{minimal surface} (local area-minimizing surface) obtained by solving Plateau's problem on the facet's bounding edges (Hass, 1991). A physical model of this minimal surface is the soap film obtained by dipping a wire frame bent in the shape of the base facet's boundary into a soap solution. This gives rise to what we now refer to as a \textit{skew realization} $\rho_{\text{sk}}$. Such a realization results to a polyhedron with facets that are curved as opposed to planar.

\section{Regular geometric $H_3$-polyhedra}
\label{sec:regularGeometricH3Polyhedra}

    The method discussed in Theorem \ref{thm:WythoffConstruction} allows one to reproduce the spherical and classical realizations of the regular abstract $H_3$-polyhedra and lets one construct non-standard realizations.

    Applying formula \eqref{eq:WythoffDimension} to the string C-groups in Table \ref{tbl:H3-polyhedra} yields six abstract polyhedra with non-zero Wythoff dimension: $\{3, 5\}^*120$, $\{5, 3\}^*120$, $\{5, 5\}^*120$, $\{6, 5\}^*120_c$, $\{10, 3\}^*120_b$, $\{10, 5\}^*120_b$. These realizable polyhedra have a 1-dimensional Wythoff space for either representation $\varphi_1$, $\varphi_2$ and, consequently, will give rise to 12 spherical and 12 non-spherical (convex, star, or skew) realizations. The resulting geometric polyhedra are rendered as solid figures using Wolfram Mathematica (2018) and presented in Tables \ref{tbl:35GeometricPolyhedra}--\ref{tbl:55GeometricPolyhedra}, \ref{tbl:65GeometricPolyhedra}--\ref{tbl:105GeometricPolyhedra}. The number of vertices $v$, edges $e$, and facets $f$ of these polyhedra are also indicated in the tables.

\begin{table}[!htb]
        \begin{tabular}{c|ccc}
            \hline
            $\varphi_{i}$ & $\rho_{\text{sp}}$ & $\rho_{\text{co}}/\rho_{\text{st}}$ \\
            \hline
            \\[1mm]
            $\varphi_{1}$ & \parbox[c]{1.3in}{\includegraphics[width = 1.3in]{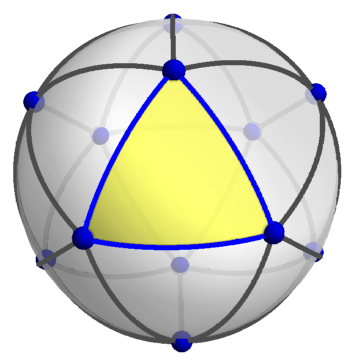}} & \parbox[c]{1.3in}{\includegraphics[width = 1.3in]{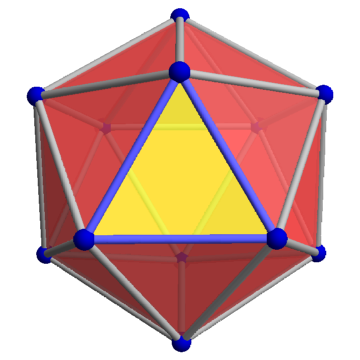}} \\
            & (\textit{a}) spherical icosahedron & (\textit{b}) convex icosahedron \\
            \hline
            \\[1mm]
            $\varphi_{2}$ & \parbox[c]{1.3in}{\includegraphics[width = 1.3in]{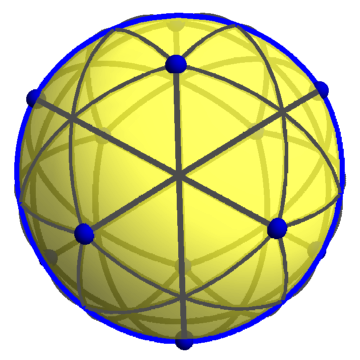}} & \parbox[c]{1.3in}{\includegraphics[width = 1.3in]{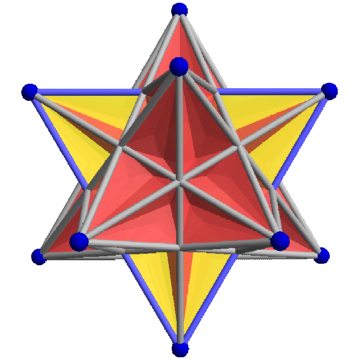}} \\
            & (\textit{c}) spherical great icosahedron & (\textit{d}) star great icosahedron \\
            \hline
        \end{tabular}
        \caption{Full rank geometric realizations of $\{3, 5\}^*120$ ($v = 12$, $e = 30$, $f = 20$) with base facet and its boundary highlighted. \label{tbl:35GeometricPolyhedra}}
    \end{table}
    
\begin{table}[!htb]
\noindent\makebox[\textwidth]{
        \begin{tabular}{c|ccc}
            \hline
            $\varphi_{i}$ & $\rho_{\text{sp}}$ & $\rho_{\text{co}}/\rho_{\text{st}}$ \\
            \hline
            \\[1mm]
            $\varphi_{1}$ & \parbox[c]{1.3in}{\includegraphics[width = 1.3in]{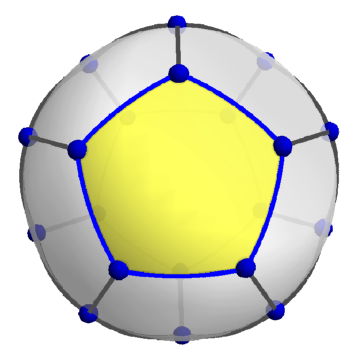}} & \parbox[c]{1.3in}{\includegraphics[width = 1.3in]{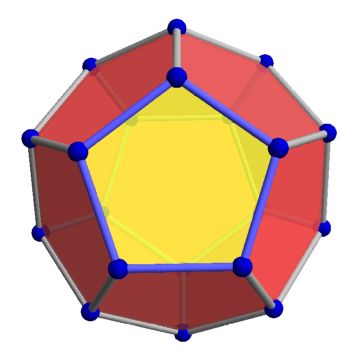}} \\
            & (\textit{a}) spherical dodecahedron & (\textit{b}) convex dodecahedron \\
            \hline
            \\[1mm]
            $\varphi_{2}$ & \parbox[c]{1.3in}{\includegraphics[width = 1.3in]{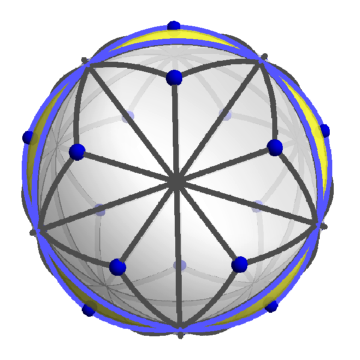}} & \parbox[c]{1.3in}{\includegraphics[width = 1.3in]{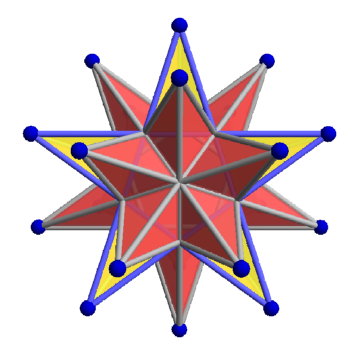}} \\
            & (\textit{c}) spherical great & (\textit{d}) star great \\
            & stellated dodecahedron & stellated dodecahedron \\
            \hline
        \end{tabular}}
        \caption{Full rank geometric realizations of $\{5, 3\}^*120$ ($v = 20$, $e = 30$, $f = 12$) with base facet and its boundary highlighted. \label{tbl:53GeometricPolyhedra}}
    \end{table}
    
\begin{table}[!htb]
\noindent\makebox[\textwidth]{
        \begin{tabular}{c|ccc}
            \hline
            $\varphi_{i}$ & $\rho_{\text{sp}}$ & $\rho_{\text{co}}/\rho_{\text{st}}$ \\
            \hline
            \\[1mm]
            $\varphi_{1}$ & \parbox[c]{1.3in}{\includegraphics[width = 1.3in]{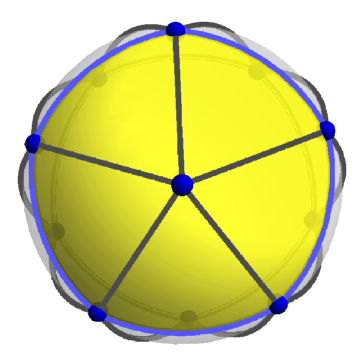}} & \parbox[c]{1.3in}{\includegraphics[width = 1.3in]{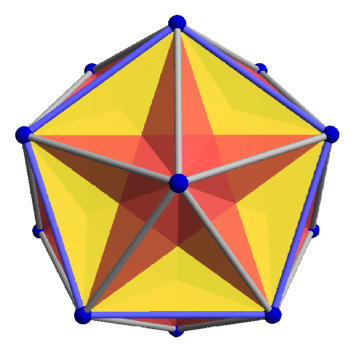}} \\
            & (\textit{a}) spherical great & (\textit{b}) star great \\
            & dodecahedron & dodecahedron \\
            \hline
            \\[1mm]
            $\varphi_{2}$ & \parbox[c]{1.3in}{\includegraphics[width = 1.3in]{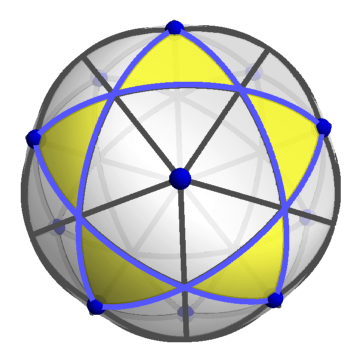}} & \parbox[c]{1.3in}{\includegraphics[width = 1.3in]{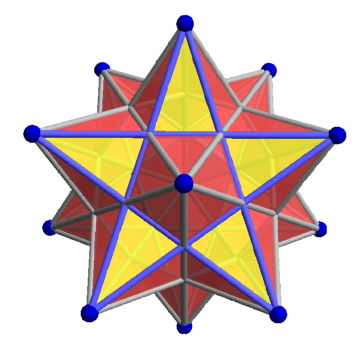}} \\
            & (\textit{c}) spherical small & (\textit{d}) star small \\
            & stellated dodecahedron & stellated dodecahedron \\
            \hline
        \end{tabular}}
        \caption{Full rank geometric realizations of $\{5, 5\}^*120$ ($v = 12$, $e = 30$, $f = 12$) with base facet and its boundary highlighted. \label{tbl:55GeometricPolyhedra}}
    \end{table}

\begin{table}[!htb]
\noindent\makebox[\textwidth]{
        \begin{tabular}{c|ccc}
            \hline
            $\varphi_{i}$ & $\{6, 5\}^{*}120_c$ & $\{10, 3\}^{*}120_b$ & $\{10, 5\}^{*}120_b$ \\
            \hline
            \\[1mm]
            $\varphi_{1}$ & \parbox[c]{0.8in}{\includegraphics[width = 0.8in]{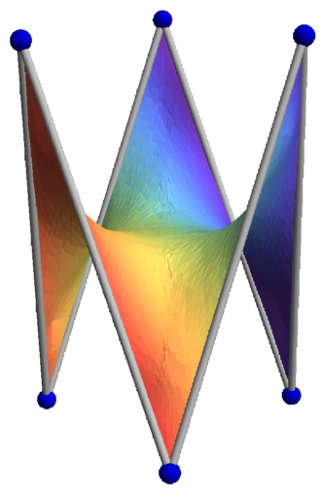}} & \parbox[c]{1.2in}{\includegraphics[width = 1.2in]{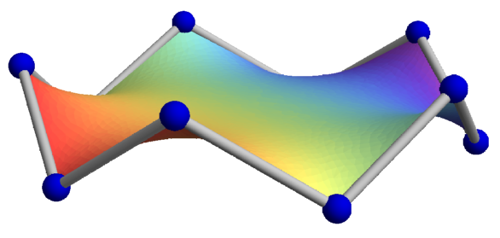}} & \parbox[c]{1in}{\includegraphics[width = 1in]{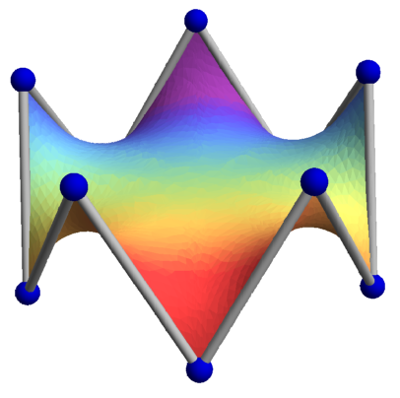}} \\
            & (\textit{a}) & (\textit{b}) & (\textit{c}) \\
            \hline
            \\[1mm]
            $\varphi_{2}$ & \parbox[c]{1.2in}{\includegraphics[width = 1.2in]{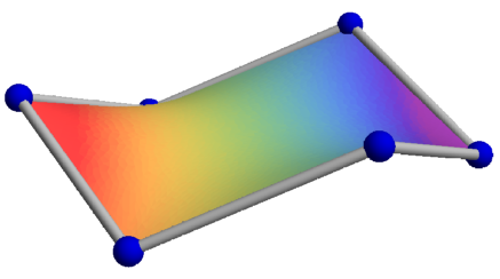}} & \parbox[c]{0.9in}{\includegraphics[width = 0.9in]{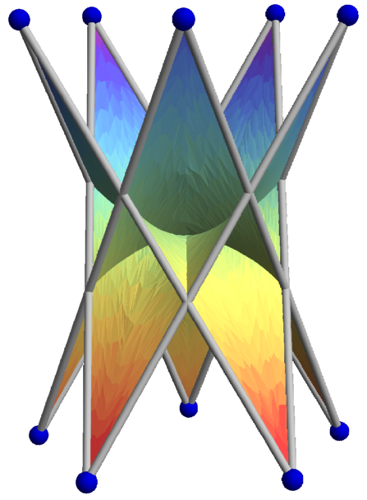}} & \parbox[c]{1.2in}{\includegraphics[width = 1.2in]{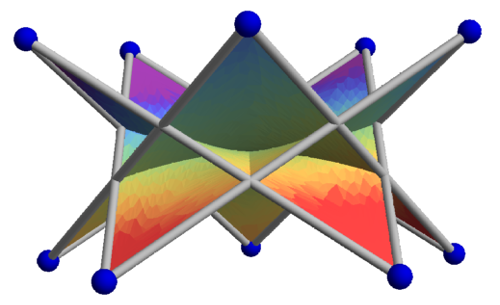}} \\
            & (\textit{d}) & (\textit{e}) & (\textit{f}) \\
            \hline
        \end{tabular}}
        \caption{Full rank geometric base facets of the skew realizations of $\{6, 5\}^{*}120_c$, $\{10, 3\}^{*}120_b$, $\{10, 5\}^{*}120_b$. \label{tbl:skewFacets}}
    \end{table}

    The spherical realizations correspond to covers of the unit sphere by spherical projections of planar triangles, pentagons, pentagrams, skew hexagons, and skew decagons. Some of these projected polygons cover the sphere only once (see Tables \ref{tbl:35GeometricPolyhedra}(\textit{a}), \ref{tbl:53GeometricPolyhedra}(\textit{a}), \ref{tbl:53GeometricPolyhedra}(\textit{c}), \ref{tbl:55GeometricPolyhedra}(\textit{c})) and, hence, generate a regular spherical tessellation.

    The classical realizations consist of two convex polyhedra: the \textit{icosahedron} (Table \ref{tbl:35GeometricPolyhedra}(\textit{b})) and the \textit{dodecahedron} (Table \ref{tbl:53GeometricPolyhedra}(\textit{b})) with a triangle and a pentagon, respectively, as facet; and four star polyhedra: the \textit{great icosahedron} (Table \ref{tbl:35GeometricPolyhedra}(\textit{d})), the \textit{great stellated dodecahedron} (Table \ref{tbl:53GeometricPolyhedra}(\textit{d})), the \textit{great dodecahedron} (Table \ref{tbl:55GeometricPolyhedra}(\textit{b})), and the \textit{small stellated dodecahedron} (Table \ref{tbl:55GeometricPolyhedra}(\textit{d})) with a triangle, a pentagram, a pentagon, and a pentagram, respectively, as a facet. These star polyhedra are also referred to as the \textit{stellations} of the convex icosahedron and dodecahedron and may be constructed alternatively by extending the facets of the latter until they intersect and form the facets of the former.

    To illustrate the similarities and differences between a spherical and a classical realization, we take the polyhedron $\{5, 3\}^*120$ and embed its realization under $\rho_{\text{st}}$ in Illustration \ref{ill:realization53} into its realization under $\rho_{\text{sp}}$. We present the embedded figures in Fig. \ref{fig:gsDodecahedron}(\textit{d}). We also highlighted the planar pentagram facet in the star polyhedron and its projection on the unit sphere in the spherical polyhedron. Notice how the edges in both polyhedra intersect at points which do not correspond to vertices.

    None of $\{6, 5\}^*120_c$, $\{10, 3\}^*120_b$, $\{10, 5\}^*120_b$ admit a convex or a star realization since their base geometric facets have non-coplanar bounding edges. By implementing a simple numerical iterative algorithm based on finite element method, we obtain a minimal surface as base facet of each of these polyhedra. For instance, when this algorithm is applied to $\{10, 3\}^*120_b$, we obtain the skew decagon facet in Fig. \ref{fig:103skewFacets} and the geometric polyhedron in Table \ref{tbl:103GeometricPolyhedra}(\textit{b}). The other five geometric realizations are displayed in Tables \ref{tbl:65GeometricPolyhedra}, \ref{tbl:103GeometricPolyhedra}, \ref{tbl:105GeometricPolyhedra}. Their facets, which are either skew hexagons or skew decagons, are shown in Table \ref{tbl:skewFacets}.
    
\begin{table}[!htb]
        \begin{tabular}{c|ccc}
            \hline
            $\varphi_{i}$ & $\rho_{\text{sp}}$ & $\rho_{\text{sk}}$ \\
            \hline
            \\[1mm]
            $\varphi_{1}$ & \parbox[c]{1.3in}{\includegraphics[width = 1.3in]{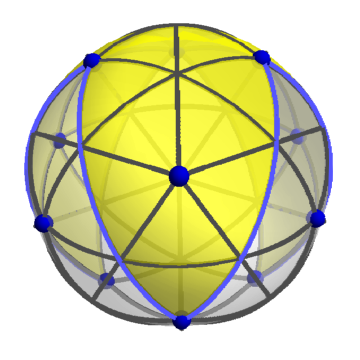}} & \parbox[c]{1.3in}{\includegraphics[width = 1.3in]{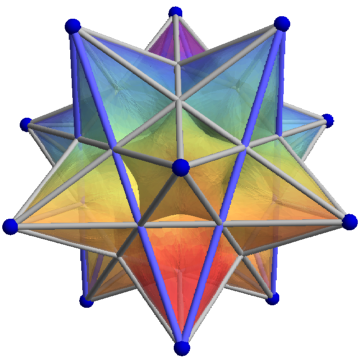}} \\
            & (\textit{a}) & (\textit{b}) \\
            \hline
            \\[1mm]
            $\varphi_{2}$ & \parbox[c]{1.3in}{\includegraphics[width = 1.3in]{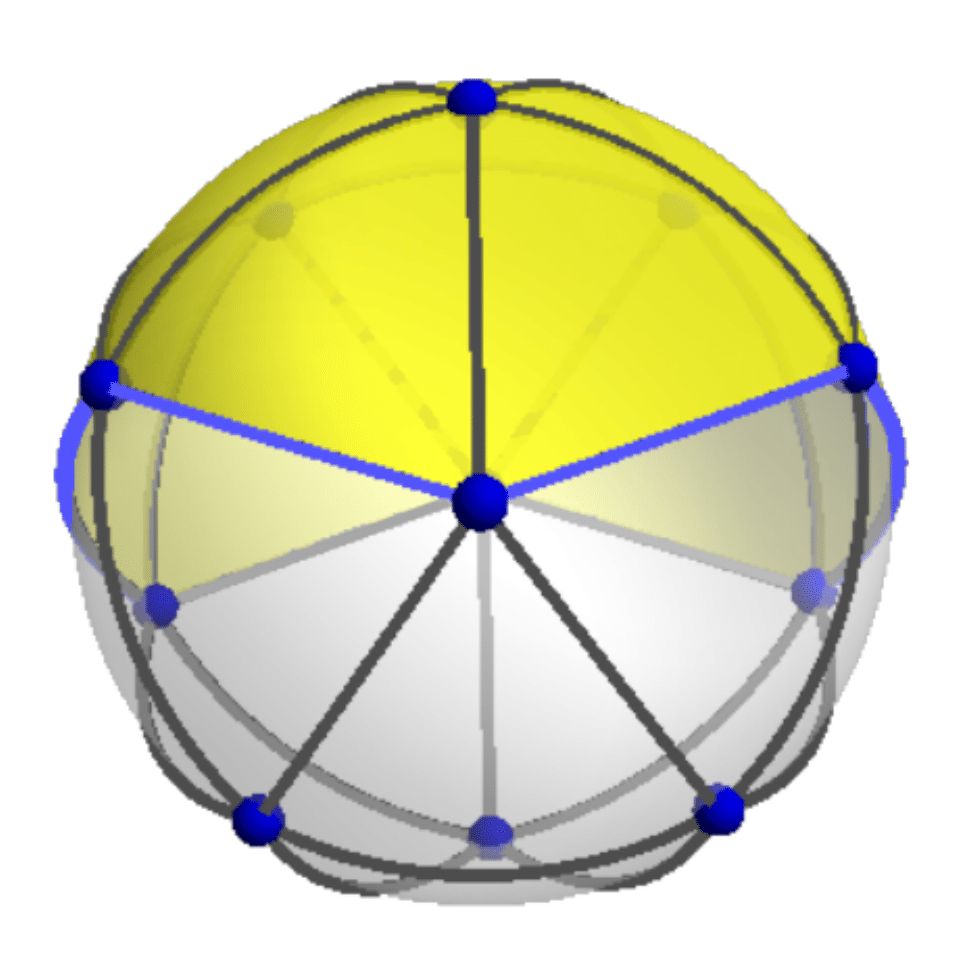}} & \parbox[c]{1.3in}{\includegraphics[width = 1.3in]{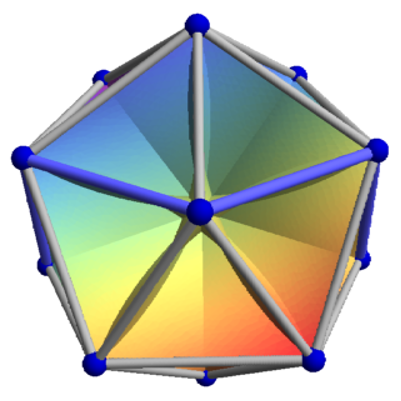}} \\
            & (\textit{c}) & (\textit{d}) \\
            \hline
        \end{tabular}
        \caption{Full rank geometric realizations of $\{6, 5\}^{*}120_c$ ($v = 12$, $e = 30$, $f = 10$) with base facet and its boundary highlighted. \label{tbl:65GeometricPolyhedra}}
    \end{table}

\begin{table}[!htb]
        \begin{tabular}{c|ccc}
            \hline
            $\varphi_{i}$ & $\rho_{\text{sp}}$ & $\rho_{\text{sk}}$ \\
            \hline
            \\[1mm]
            $\varphi_{1}$ & \parbox[c]{1.3in}{\includegraphics[width = 1.3in]{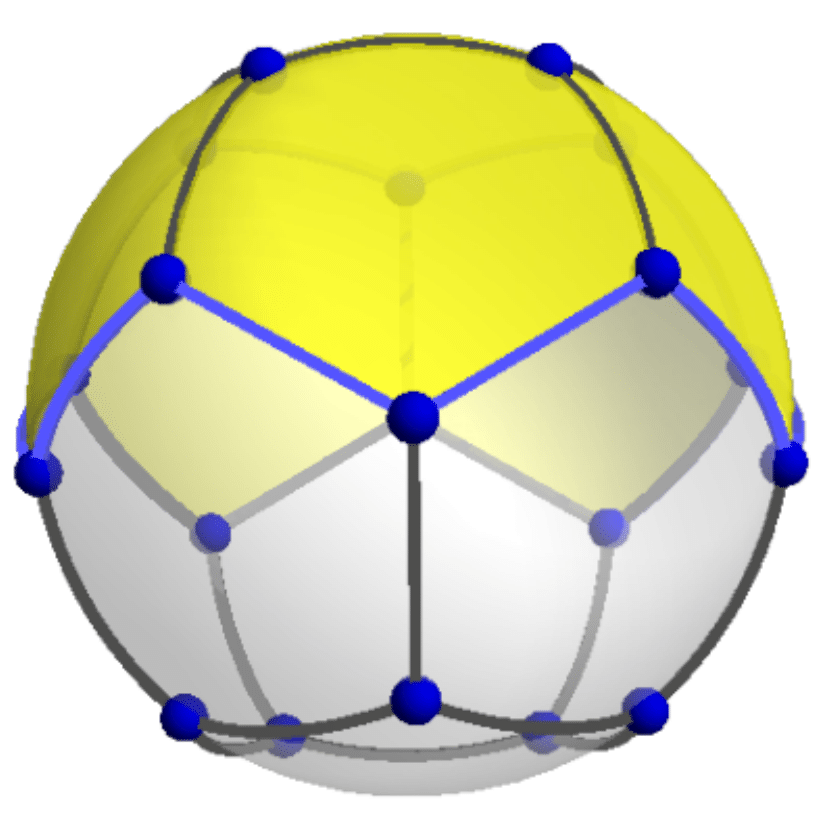}} & \parbox[c]{1.3in}{\includegraphics[width = 1.3in]{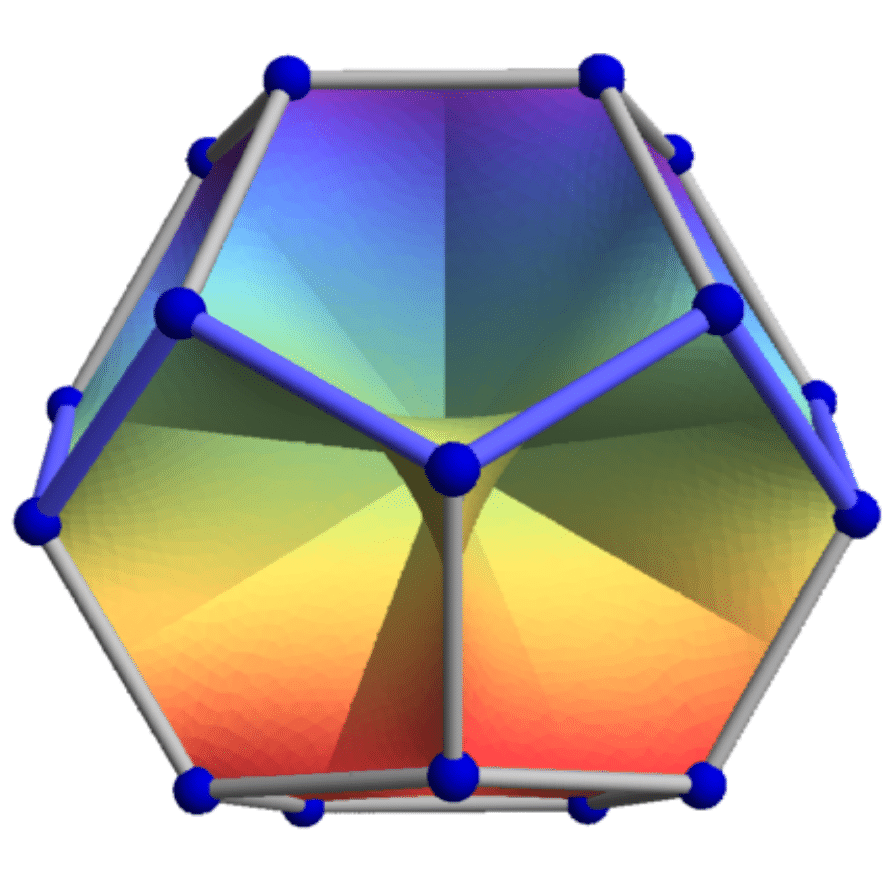}} \\
            & (\textit{a}) & (\textit{b}) \\
            \hline
            \\[1mm]
            $\varphi_{2}$ & \parbox[c]{1.3in}{\includegraphics[width = 1.3in]{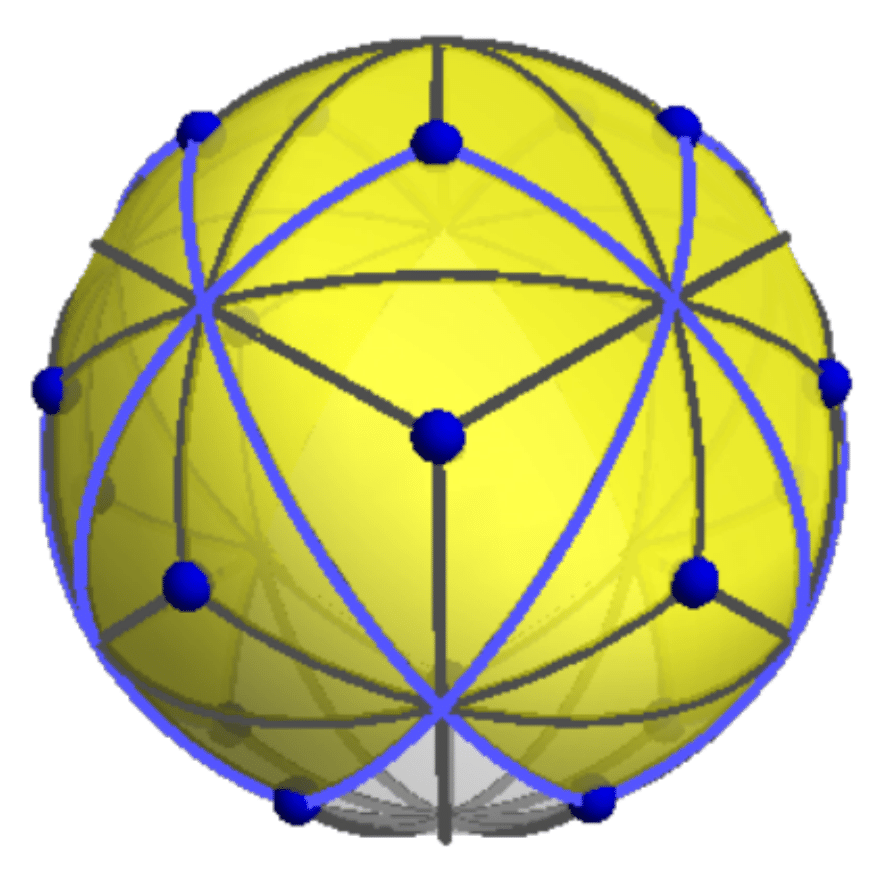}} & \parbox[c]{1.3in}{\includegraphics[width = 1.3in]{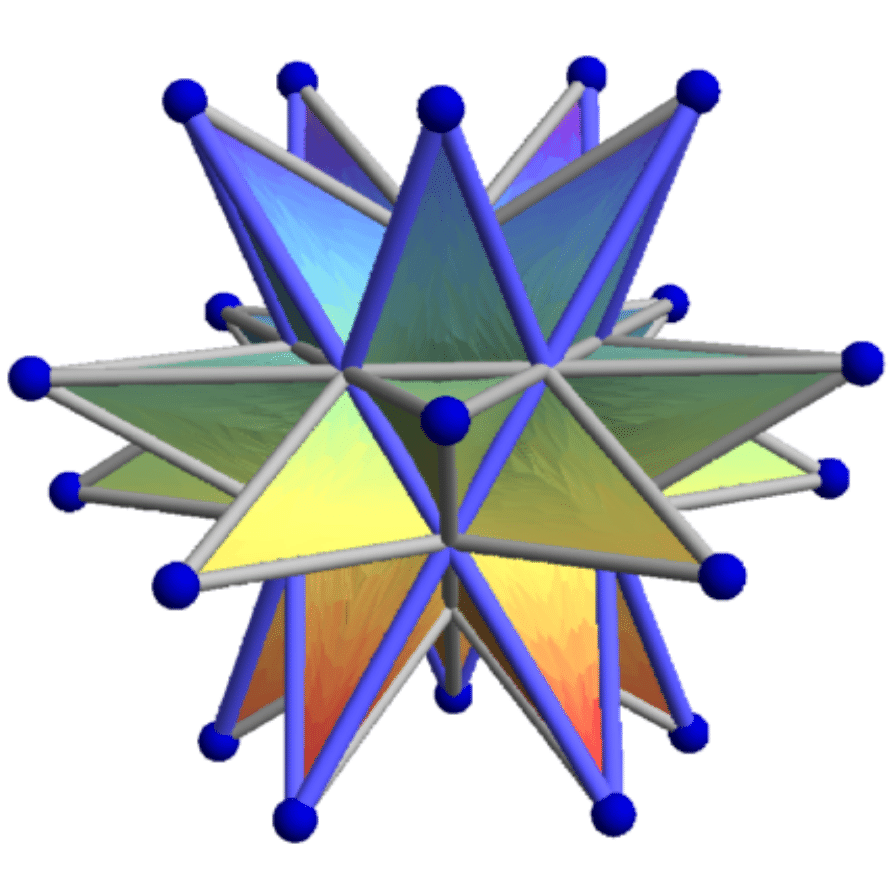}} \\
            & (\textit{c}) & (\textit{d}) \\
            \hline
        \end{tabular}
        \caption{Full rank geometric realizations of $\{10, 3\}^{*}120_b$ ($v = 20$, $e = 30$, $f = 6$) with base facet and its boundary highlighted. \label{tbl:103GeometricPolyhedra}}
    \end{table}

\begin{table}[!htb]
        \begin{tabular}{c|ccc}
            \hline
            $\varphi_{i}$ & $\rho_{\text{sp}}$ & $\rho_{\text{sk}}$ \\
            \hline
            \\[1mm]
            $\varphi_{1}$ & \parbox[c]{1.3in}{\includegraphics[width = 1.3in]{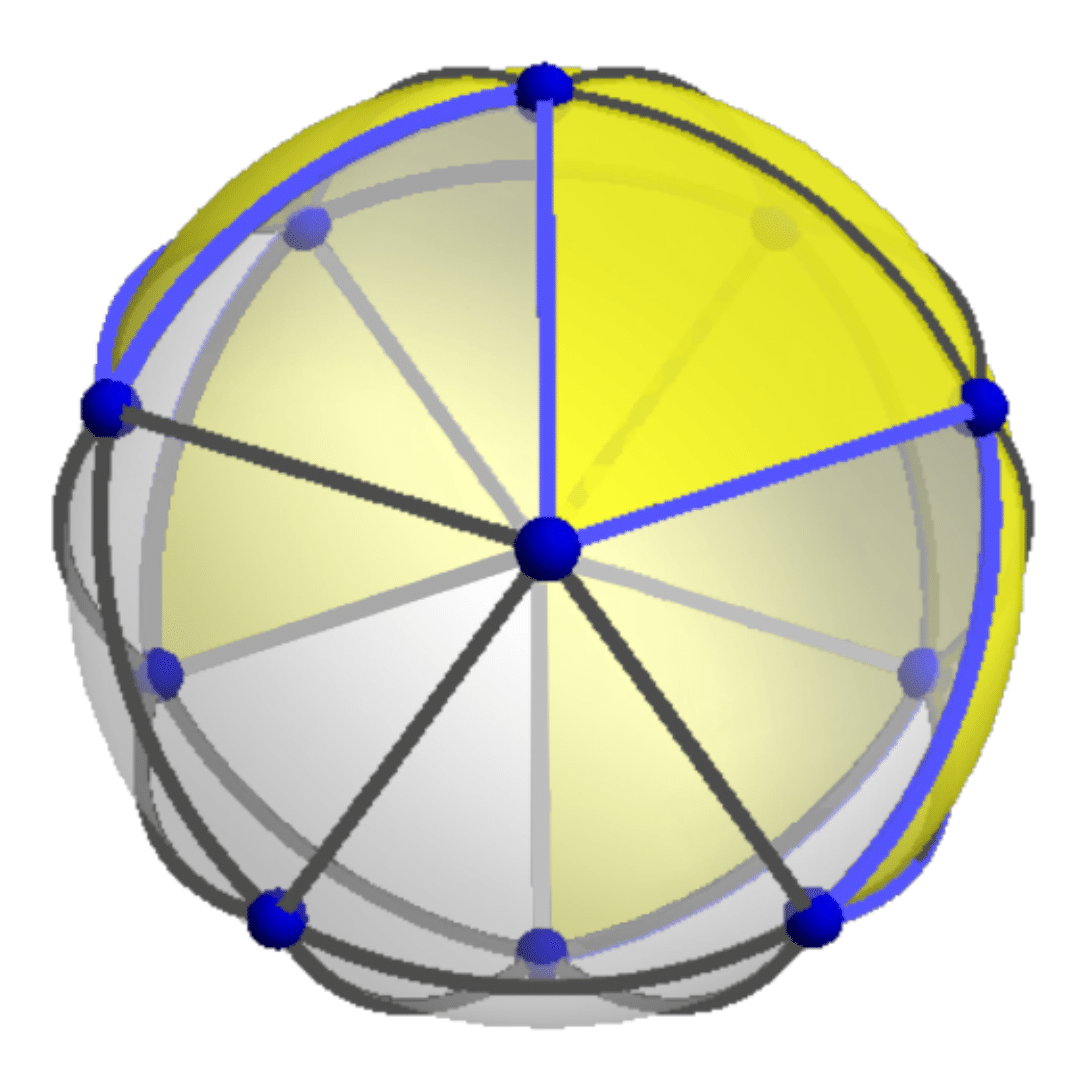}} & \parbox[c]{1.3in}{\includegraphics[width = 1.3in]{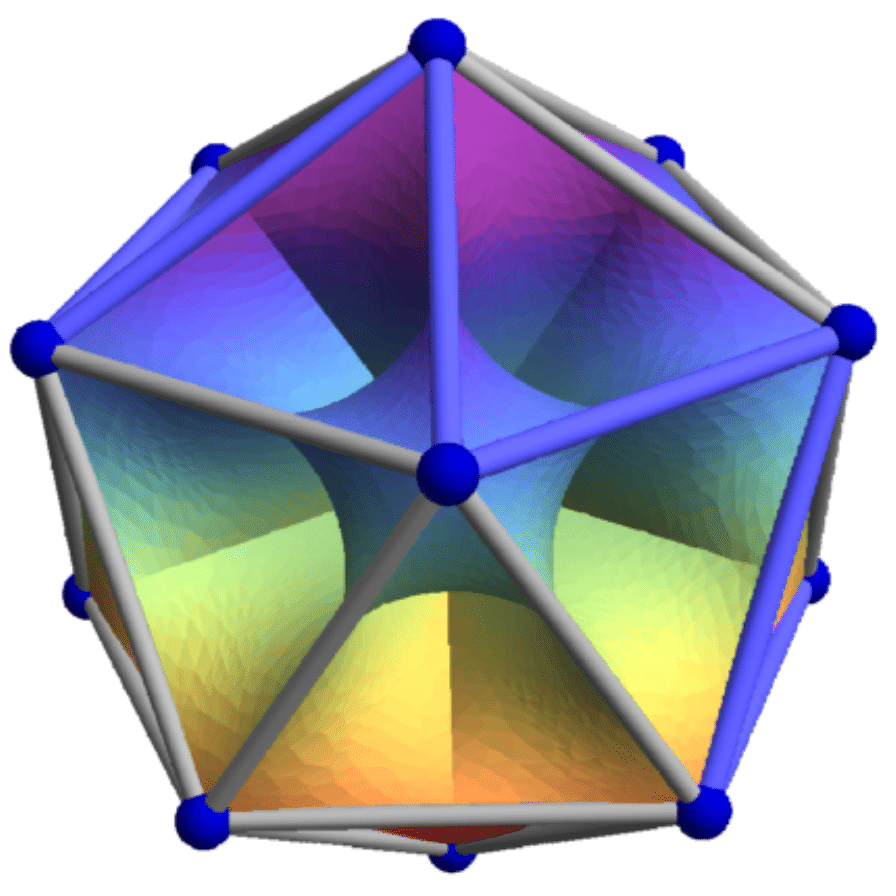}} \\
            & (\textit{a}) & (\textit{b}) \\
            \hline
            \\[1mm]
            $\varphi_{2}$ & \parbox[c]{1.3in}{\includegraphics[width = 1.3in]{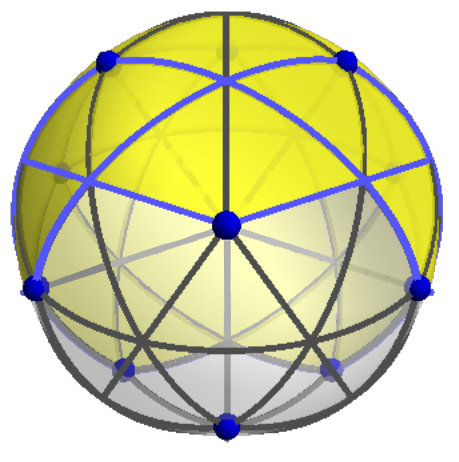}} & \parbox[c]{1.3in}{\includegraphics[width = 1.3in]{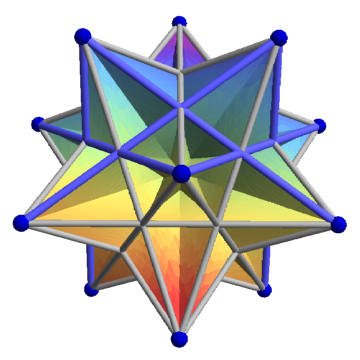}} \\
            & (\textit{c}) & (\textit{d}) \\
            \hline
        \end{tabular}
        \caption{Full rank geometric realizations of $\{10, 5\}^{*}120_b$ ($v = 12$, $e = 30$, $f = 6$) with base facet and its boundary highlighted. \label{tbl:105GeometricPolyhedra}}
    \end{table}

\section{Conclusion and future outlook}
\label{sec:conclusion}

    This study demonstrated a method to produce a full rank geometric realization of a regular abstract polyhedron. Existing works on realizations emphasize their algebraic aspects. For instance, the articles of McMullen (1989, 2011), McMullen \& Monson (2003), and Ladisch (2016) focus on the structure of the \textit{realization cone} or the set of all realizations of a given polytope up to congruence. In contrast, we highlighted their geometric aspects by identifying the realizations with their images as solid figures in space.

    Adapted from Wythoff construction, the method presented in this study is algorithmic in nature and, hence, well-suited for computer implementation. This was exhibited when we applied the method to regular abstract polyhedra with automorphism group isomorphic to $H_3$. The entire process involved enumerating the regular abstract $H_{3}$-polyhedra through a search algorithm in GAP and using the irreducible orthogonal representations of $H_3$ to generate the corresponding figures of their geometric realizations in Wolfram Mathematica. This allowed us to reproduce the classical convex and star polyhedra  with icosahedral symmetry, as well as, non-standard icosahedral polyhedra with minimal surfaces as facets. We reiterate that we do not limit ourselves to the families of open sets listed in \S\ref{subsec:geometricFaces} when considering a realization.

    We remark that even though we applied the method only to the regular abstract $H_3$-polyhedra, it is also applicable to other regular polyhedra and may be extended to polytopes of higher ranks. In particular, one may apply the method to regular polyhedra arising from the groups $A_n$, $B_n$, and $H_n$ (Humphreys, 1992). For future work, it is worthwhile to consider establishing an analogous construction method for the realizations of \textit{semi-regular} abstract polytopes using a version of Wythoff construction found in Monson \& Schulte (2012) as a framework.

\bigskip

{\it Jonn Angel L. Aranas, Department of Mathematics, Ateneo de Manila University, Katipunan Avenue, Loyola Heights, Quezon City 1108, Philippines. E-mail address: jonn.aranas@obf.ateneo.edu}

\medskip

{\it Mark L. Loyola, Department of Mathematics, Ateneo de Manila University, Katipunan Avenue, Loyola Heights, Quezon City 1108, Philippines. E-mail address: mloyola@ateneo.edu}

\end{document}